\documentclass[12pt]{article}

\usepackage{times}
\usepackage[centertags]{amsmath}
\usepackage{amssymb}
\usepackage{amsthm}
\usepackage[utf8]{inputenc}
\usepackage[T1]{fontenc}
\usepackage[colorlinks=true,urlcolor=red,citecolor=red,linkcolor=red]{hyperref}
\usepackage[a4paper, top=5cm, bottom=5cm, left=3.5cm, right=2.5cm]{geometry}
\usepackage[italian,english]{babel}
\usepackage{float}
\usepackage{mathtools}
\usepackage{graphicx}
\usepackage{tikz}
\usepackage{qtree}
\usepackage{tikz-qtree}
\usepackage[lined,algonl,boxed]{algorithm2e}
\usepackage{pgfplots}
\pgfplotsset{compat=1.13}
\usetikzlibrary{trees,positioning,shapes}
\usetikzlibrary{patterns}
\usepackage{verbatim}
\newtheorem{teo}{Theorem}[section]
\newtheorem{defi}[teo]{Definition}
\newtheorem{prop}[teo]{Proposition}
\newtheorem{lem}[teo]{Lemma}
\newtheorem{cor}[teo]{Corollary}
\theoremstyle{definition}
\newtheorem{oss}[teo]{Remark}

\theoremstyle{definition}
\newtheorem{ex}[teo]{Example}

\newcommand{\bs}{\boldsymbol}

\newcommand{\N}{\mathbb{N}}
\newcommand{\Z}{\mathbb{Z}}

\newenvironment{acknowledgements}%
{\null\vfill\begin{center}%
		\bfseries Acknowledgements\end{center}}%
{\vfill\null}
\newcommand{\keywords}[1]{\emph{Keywords:} #1}
\newcommand{\MSC}[1]{\emph{Mathematics Subject Classification 2010:} #1}

\newcounter{lastnote}

\title{The tree of good semigroups in $\N^2$ and a generalization of the Wilf conjecture} 
\author{N. Maugeri\thanks{\emph{e-mail: nicola.maugeri.1992@gmail.com} \newline \hspace*{1.8em}Part of this work was done while the first author was visiting the Universities of Almeria and Granada supported by \newline \hspace*{1.8em} the project MTM2014-55367-P, which is funded by Ministerio de Economía y Competitividad and Fondo\newline \hspace*{1.8em} Europeo de Desarrollo Regional FEDER, and by the Junta de Andalucía Grant Number FQM-343.}, G. Zito\thanks{ \emph{e-mail: giuseppezito@hotmail.it}
\newline Both the authors was funded by the project "Proprietà algebriche locali e globali di anelli associati a curve e ipersuperfici" PTR 2016-18 - Dipartimento di Matematica e Informatica - Università di Catania".}
}

\date{}

\begin{document}
\linespread{1,0}
\maketitle
\begin{abstract}
\noindent
Good subsemigroups of $\N^d$ have been introduced as the most natural generalization of numerical ones. Although their definition arises by taking in account the properties of  value semigroups of analytically unramified rings (for instance the local  rings of an algebraic curve),  not all good semigroups can be obtained as value semigroups, implying that they can be studied as pure combinatorial objects. In this work, we are going to introduce the definition of length and genus for good semigroups in $\N^d$.
For $d=2$, we show how to count  all the local good semigroups with a fixed genus through the introduction of the tree of local  good subsemigroups of $\N^2$, generalizing the analogous concept introduced in the numerical case. Furthermore, we study the relationships between these elements and others previously defined in the case of good semigroups with two branches, as the type and the embedding dimension.
Finally we extend the well known Wilf conjecture to good semigroups in $\N^2$  drawing conclusions about its validity in this more general context.
\end{abstract}
\\
\MSC{13A18, 14H99, 13H99, 20M25.}

\section*{Introduction}
The study of good semigroups was formerly motivated by the fact that they are the value semigroups of one-dimensional analytically unramified rings (such as the local rings of an algebraic curve). The definition appeared the first time in \cite{anal:unr} and these objects were widely studied in several works \cite{two:danna, symm:cdk, multibranch:delgado, symm:delgado, semig:garcia}.
In \cite{anal:unr}, the authors proved that the class of good semigroups is actually larger than the one of value semigroups. Thus, such semigroups can be seen as a natural generalization of numerical semigroups and can be studied using a more combinatorial approach without necessarily referring to the ring theory context. In recent works \cite{DAGuMi}, \cite{emb-NG}, \cite{type:good}, some notable elements and properties of numerical semigroups have been generalized to the case of good semigroups.
The main purpose of this work is to generalize the definitions of \emph{length} and \emph{genus} of an ideal of a numerical semigroup to the case of good ideals of a good semigroup studying also the relationships between them and the other objects defined in the previous works in the case of subsemigroups of $\N^2$.\\
If $R$ is an analytically unramified ring, the value semigroup $v(R)=\{\bs{v(r)}$ $|$ $r$ is not a zerodivisor of $R\}$ is a good semigroup \cite{canonical:danna}.
If $I$ is a relative good ideal of $R$, the extension $I\subseteq \bar{I}$ is of finite type and the conductor ideal is $C(I)=I:\bar{I}$, where both closure and colon operations are considered in the ring of total fractions. For a fixed $\bs{\alpha}\in \Z^d$, we denote by $I(\bs{\alpha})=\{r\in R\hspace{0.1cm}|\hspace{0.1cm} v(r)\geq \bs{\alpha}\}$.
If $\bs{c}(v(I))$ is the conductor of the ideal $v(I)$ of $v(R)$, we have $v(I:\bar{I})=v(I(\operatorname{\bs{c}}(I)))$.
In the one-branch case, given a relative ideal $I$ of $R$, we have that the length of the $R$-module $l_R(I/C(I))$ is equal to $n(v(I))$, where $n(v(I))$ is the cardinality of the set of small elements of the numerical semigroup $v(I)$. For this reason, given a relative ideal $E$ of a numerical semigroup $S$, it is natural to call \emph{length} of $E$ the number $n(E)$. On the other hand, the \emph{genus} of $E$ is defined as the number of gaps in $E$ and it is denoted by $g(E)$. It is straightforward that $g(E)+n(E)=\operatorname{c}(E)$.\\
In Section \ref{sec1} we recall the definition of good semigroup and we fix the basic notations. In \cite{canonical:danna} it is defined a function of distance $d$ between relative good ideals in a good semigroup $S$ and it is proved that if $S$ is a semigroup of values of a ring $R$, given a good relative ideal $I$ of $R$, we have $l_R(I/C(I))=d(v(I)\setminus v(C(I)))$. For this reason, taking into account the additivity of the function $d$, we can generalize in a natural way the definition of length and genus to the case of good ideals of $\N^d$ (not necessarily in case of the semigroup of values of a ring) as it was done in \cite{zito2018arf} for Arf semigroups. Given a relative good ideal $E\subseteq S$, we define respectively length and genus of $E$ as $l(E)=d(E\setminus E(\operatorname{\bs{c}}(E)))$ and $g(E)=d(\N^d\setminus E(\operatorname{\bs{c}}(E))$. We conclude the section giving a slightly different version of an explicit method to compute length and genus introduced in \cite{symm:delgado}. In \cite{zito2018arf}, it  was computed the number of good Arf semigroups with $n$ branches having a fixed genus using the untwisted  multiplicity trees defined in \cite{giuseppe:Arf}. Our aim is to obtain a similar result for a general good semigroup. In \cite{Bras-Amoros2008}, it is presented a method to compute all numerical semigroups up to  a fixed genus building a tree where each new level is obtained removing  minimal generators larger than the Frobenius number from the semigroups of the previous level. In Section \ref{sec2}, we repeat the same idea for good semigroups $S\subseteq \N^2$; in this case, the \emph{tracks} of the good semigroup, defined in \cite{emb-NG}, will have the role that minimal generators played in case of numerical one. In order to do this, we prove that every good semigroup of genus $g$ can be obtained removing a track from one of genus $g-1$ (Theorem \ref{teo1}) and that by removing a track from a good semigroup of genus $g$ we obtain a good semigroup of genus $g+1$ (Theorem \ref{genuson}).
Then, we  explain how to build the tree of good semigroups, underlining the differences with the numerical case. We report the results regarding the computation of the number of local good semigroups with a fixed genus up to genus $27$, produced with an algorithm written in "GAP" \cite{GAP4} using the package "NumericalSgps" \cite{GAPNumericalSgps}. In Section \ref{sec3} we study the relationships between genus, length and other notable elements of a good semigroup. 
In \cite{froberg1986numerical}, it was proved the inequality $c(S)\leq (t(S)+1)l(S)$ for numerical semigroups. The \emph{type} of a good semigroup was originally defined in \cite{anal:unr} for the good semigroups such that $S-M$ is a relative good ideal of $S$ and it was recently generalized in \cite{type:good}. We conclude the paper asking if this inequality holds for good semigroups with respect to the generalized version of definition of type.
We definitely prove that length, genus and type satisfy the relationships $t(S)+l(S)-1\leq g(S)\leq t(S)l(S)$ also in the case of good semigroups (Proposition \ref{boundsup} and Corollary \ref{boundinf}).
To conclude the paper we observe that the definitions and the algorithm given in the previous section give us the possibility to introduce an analogue of the Wilf conjecture for good semigroups. In this case, we found counterexamples for the conjecture (Example \ref{countwilf}) but the problem remains open for good semigroups which are value semigroups of a ring.

\section{Length and Genus of a Good Ideal}
\label{sec1}
We begin recalling the definition of good semigroup introduced in \cite{anal:unr}.
\begin{defi}
\label{E1}
A submonoid $S$ of $(\N^d,+)$ is a \emph{good semigroup} if it satisfies the following properties:
\begin{itemize}
\item[(G1)] If $\bs{\alpha}=(\alpha_1,\ldots,\alpha_d), \bs{\beta}=(\beta_1,\ldots,\beta_d)$ belong to $S$, then $\min(\bs{\alpha};\bs{\beta}) :=(\min\{\alpha_1,\beta_1\},\ldots, \min\{\alpha_d,\beta_d\})\in S$;
\item[(G2)] There exists $\bs{\delta} \in \N^d$ such that $\bs{\delta}+\N^d \subseteq S$;
\item[(G3)] If $\bs{\alpha}=(\alpha_1,\ldots,\alpha_d),\bs{\beta}=(\beta_1,\ldots,\beta_d)\in S$; $\bs{\alpha}\neq \bs{\beta}$ and $\alpha_i=\beta_i$ for some $i\in\{1,\ldots,d\}$; then there exists $\bs{\epsilon} \in S$ such that $\epsilon_i>\alpha_i=\beta_i$ and $\epsilon_j\geq \min\{\alpha_j,\beta_j\}$ for each $j\neq i$ (if $\alpha_j\neq \beta_j$, the equality holds).
\end{itemize}
\end{defi}
Furthermore, we always assume to work with a \emph{local} good semigroup $S$, that is, if $\bs{\alpha}=(\alpha_1,\ldots,\alpha_d)\in S$ and $\alpha_i=0$ for some $i\in\{1,\ldots,d\}$, then $\bs{\alpha}=\bs{0}=(0,\ldots,0)$.
As a consequence of property (G2), the element $\bs{c(S)}=\min\{\bs{\delta}\hspace{0.1cm}|\hspace{0.1cm} S\supseteq \bs{\delta}+\N^d\}$ is well defined and it is called the \textit{conductor} of the good semigroup. The element $\bs{f(S)}=\bs{c(S)}-\bs{1}$, where $\bs{1}=(1,\ldots,1)$, is called the \textit{Frobenius vector} of the good semigroup. On the elements of $\N^d$ is naturally defined a partial order relation induced by the usual order on $\N$. Hence, given $\bs{\alpha}=(\alpha_1,\ldots\alpha_d)$, $\bs{\beta}=(\beta_1,\ldots,\beta_d)\in S$, we write $\bs{\alpha}\leq \bs{\beta}$ if $\alpha_i\leq\beta_i$ for all $i\in\{1,\ldots,d\}$. It is possible to show that a good semigroup can be described by a finite set of elements, this set is denoted by $\operatorname{Small}(S)=\{\bs{\alpha}\in S\hspace{0.1cm}|\hspace{0.1cm} \bs{\alpha}\leq \bs{c(S)}\}$ and its elements are called the \textit{small elements} of $S$. The good semigroup having small elements $\{\bs{0}, \bs{1}\}$ is a local good semigroup containing all the other ones, we denote it by $\N^d(1,\ldots,1)$.

Following the notations reported in \cite{canonical:danna}, we recall some definitions.\\
Let $S$ be a good semigroup. If $E\subseteq \Z^d$ is such that $E+S\subseteq E$ and $\bs{\alpha} +E\subseteq S$
for some $\bs{\alpha}\in S$, then $E$ is called a \textit{relative ideal} of $S$. A relative ideal of $S$ needs
not to satisfy the properties (G1) and (G3) of good semigroups.
A relative ideal $E$ that does satisfy properties (G1) and (G3) will be called a \textit{relative good ideal}.
Given a good semigroup $S$ and a relative good ideal $E$, two elements $\bs{\alpha},\bs{\beta}$ in $E$ are \emph{consecutive} if there are no elements $\bs{\gamma} \in E$ such that $\bs{\alpha}<\bs{\gamma}<\bs{\beta}$.\\
An ordered sequence of $n+1$ elements in $E$: $$\bs{\alpha}=\bs{\alpha^{(0)}}<\bs{\alpha^{(1)}}<\ldots<\bs{\alpha^{(n)}}$$  is called a \emph{chain} of \emph{length} $n$ in $E$; furthermore it is called \emph{saturated} in $E$ if all its elements are consecutive in $E$.\\
In \cite{canonical:danna} it is proved that, if $\bs{\alpha},\bs{\beta}\in E$ with $\bs{\alpha}<\bs{\beta}$, then all the saturated chains between $\bs{\alpha}$ and $\bs{\beta}$ have the same length. This common length is denoted by $d_E(\bs{\alpha},\bs{\beta})$.\\
Given two relative good ideals $E\supseteq F$, we denote by $\bs{e(E)},\bs{e(F)}$ the minimal elements of $E$ and $F$ respectively; for $\bs{\alpha}\in \N^d$ big enough one has that $\bs{\alpha}\in E$ and $\bs{\alpha}\in F$. Furthermore, in [7] it is proved that the number $d_E(\bs{e(E)},\bs{\alpha})-d_F(\bs{e(F)},\bs{\alpha})$ is independent of the choice of $\bs{\alpha}$. So we can define
$$d(E\setminus F):=d_E(\bs{e(E)},\bs{\alpha})-d_F(\bs{e(F)},\bs{\alpha}).$$
Given $\bs{\alpha} \in \N^d$, we denote by $E(\bs{\alpha})=\{\bs{\beta}\in E: \bs{\beta}\geq \bs{\alpha}\}$.
The function $d(-\setminus-)$ satisfies the following properties: 
\begin{prop}\cite{canonical:danna}
\label{propBDF}
\begin{enumerate}
    \item[1)] If $E\supseteq F\supseteq G$ are good relative ideals of S, then we have $d(E \setminus G) = d(E \setminus F) + d(F \setminus G)$.
    \item[2)] If $E\supseteq F$ are good relative ideals of S, then $d(E \setminus F) = 0$ if and only if $E=F$.
    \item[3)] Let us consider $E$, a good relative ideal of $S$ and $\bs{\alpha}\in \Z^d$. If $\bs{\alpha}^i=\bs{\alpha}+\bs{e}^i$, where $e_j^i=0$ if $j\neq i$ and $e_i^i=1$, then $d(E(\bs{\alpha})\setminus E(\bs{\alpha}^i))\leq 1$, where  the equality holds if and only if $\{\bs{\beta}\in E \hspace{0.1cm}| \beta_i=\alpha_i \text{ and } \beta_j\geq \alpha_j, \text{ if } j\neq i\}\neq \emptyset$.
\end{enumerate}
\end{prop}
Given a good ideal $E$, there  always exists a minimal element $\bs{c(E)}$ such that if $\bs{\alpha}\geq \bs{c(E)}$, then $\bs{\alpha}\in E$. The element $\bs{c(E)}$  is called the \emph{conductor} of the ideal $E$ and the ideal $C(E):=E(\bs{c(E)})$ is called the \emph{conductor ideal} of $E$.
If $\bs{c(E)}=(c_1,\ldots, c_d)$, we denote by $c_E:=c_1+\ldots+c_d$.

\begin{defi}
Given a good ideal $E$ of a good semigroup $S\subseteq \N^d$, we define the \emph{genus} of $E$ as the number $g(E)=d(\N^d\setminus E)$ and the length of $E$ as the number $l(E)=d(E\setminus C(E))$. In particular $g(S)=d(\N^d\setminus S)$ and $l(S)=d(S\setminus C(S))$.
\end{defi}

\begin{oss}
\label{genalt}
By Proposition \ref{propBDF}.1), $d(\N^d\setminus C(E))=d(\N^d\setminus E)+d(E\setminus C(E))$. Thus, we can write: $g(E)=c_E-l(E)$.
\end{oss}

Now we want to introduce a useful formula for the computation of the genus.
As in \cite{Carvalho:Semiring} and in \cite{emb-NG}, we will work on the equivalent structure of semiring $\Gamma_S$ in order to simplify the notation.\\ 
We consider the one-point compactification of the topological space $\N$, hence we set $\overline{\N}=\N\cup\{\infty\}$ and we extend the natural order and the sum over $\N$ to $\overline{\N}$, setting respectively, $a<\infty$ for all $a\in \N$ and  $x+\infty=\infty+x=\infty$.\\

Let $S\subseteq \N^d$ be a good semigroup with $\bs{c(S)}=(c_1,\ldots,c_d)$. Denoting by $I=\{1,\ldots d\}$ and taking $U\subseteq I$ we introduce the sets $S^{\infty(U)}\subseteq\overline{\N}^d$ as the sets containing the limits of elements of $S$ in $\N^d$ with fixed projection $\alpha_i$, for all $i\in U$. 
More explicitly we say that the vector $(\alpha_{1},\ldots,\alpha_{d})\in S^{\infty(U)}$ if and only if $\alpha_{j}=\infty$ for all $j\notin U$  and there exists $\bs{\beta} \in S$ such that $\beta_j=\alpha_j$ if $j\in U$ and $\beta_j= c_j$ if $j\notin U$.\\

We set
$$S^{\infty}:=\bigcup_{U\subsetneq I} S^{\infty(U)},$$
and we notice that $S^{\infty(\emptyset)}=\{\bs{\infty}=(\infty,\ldots,\infty)\}$ and $S^{\infty(I)}=S$.
Hence we denote by $\Gamma_S:=S\cup S^{\infty}$ the closure of $S$ in $\N^d$. Given $\bs{\alpha}=(\alpha_1,\ldots,\alpha_d),$ $\bs{\beta}=(\beta_1,\ldots,\beta_d)\in \Gamma_S$, we introduce the operations $$\bs{\alpha}\oplus\bs{\beta}:=\min\{\bs{\alpha},\bs{\beta}\}=(\min\{\alpha_1,\beta_1\},\ldots,\min\{\alpha_d,\beta_d\}),$$
$$\bs{\alpha}\odot\bs{\beta}:=\bs{\alpha}+\bs{\beta}.$$
It is easy to prove that $(\Gamma_S,\oplus,\odot)$ is a semiring.

Given a subset $E$ of a good semigroup $S\subseteq \N^d$, we denote by \[E_j=\{e_j\hspace{0.1cm}|\hspace{0.1cm}(e_1,\ldots,e_d)\in E\},\] the $j$th projection of $E$. We introduce the following numbers:
\begin{eqnarray*}
l_S(E_j)&=&|\{\alpha\in E_j\hspace{0.1cm}|\hspace{0.1cm}\alpha_j< c_j\}|,\\ 
g_S(E_j)&=&|\{\alpha\in \N\setminus E_j\hspace{0.1cm}|\hspace{0.1cm}\alpha_j< c_j\}|.
\end{eqnarray*}
\begin{oss} It is easy to notice that, given $U \subseteq I$, the subset of $S$ defined as 
$$ E^{U}=\{ \bs{\alpha}=(\alpha_1,\ldots,\alpha_d) \in S \hspace{0.1cm}|\hspace{0.1cm} \alpha_j \geq c_j \hspace{0.1cm} \forall j\notin U\},$$
is actually a good relative ideal of $S$.
Furthermore, considering $j \in U$, we have that $j$th projection $S_j^{\infty(U)}$ of $S^{\infty(U)}$ is equal to $E^{(U)}_j$.
\end{oss}
In pursuit of  an idea reported in \cite{symm:delgado}, with the following recursive formulas we can reduce to compute the length and the genus of a good semigroup by only considering  small elements of numerical ones.
\begin{prop}
\label{length}
Given a good semigroup $S\subseteq \N^d$, if we denote by $U_i=\{i,\ldots,d\}\subseteq I$, we have: 
\begin{eqnarray*}
l(S)&=&\sum_{i=1}^d l_S(S^{\infty(U_i)}_i),\\
g(S)&=&\sum_{i=1}^d g_S(S^{\infty(U_i)}_i).
\end{eqnarray*}
\end{prop}
\begin{proof}
Let us consider the chain from $\bs{0}$ to $\bs{c(S)}=(c_1,\ldots,c_d)$, where the elements have the form $(i_1,\ldots,i_j,\ldots,i_d)$ with $0\leq i_j\leq c_j$, ordered with respect the lexicographical order on $\N^d$.\\
We consider the elements of the chain having the form $(i_1,0,\ldots,0)$ with $i_1\in \{0,\ldots,c_1-1\}$, as a consequence of  Proposition \ref{propBDF}.3), we have:
\begin{eqnarray*}
d(S((i_1,0,\ldots,0))\setminus S((i_1+1,0,\ldots,0)))=1 \Leftrightarrow\\
\{\bs{\beta}\in S\hspace{0.1cm}|\hspace{0.1cm} \beta_1=i_1\hspace{0.1cm},\hspace{0.1cm} \beta_k\geq 0,k\neq 1\hspace{0.1cm}\}\neq \emptyset \Leftrightarrow\\
\text{there exists an element of the form } (i_1,\alpha_2,\ldots,\alpha_d)\in S.  
\end{eqnarray*}
Thus, by Proposition \ref{propBDF}.1), we have:
$$d(S((0,0,\ldots,0))\setminus S((c_1,0,\ldots,0)))=l(S_1)=l_S(S^{\infty(U_1)}_1)$$
Now, let us consider the elements of the chain of the form $(c_1,\ldots,c_{j-1},i_j,\ldots,0)$ with $i_j\in \{1,\ldots,c_j\}$. We observe that $(c_1,\ldots,c_{j-1},i_j,\ldots,0)\in S$ if and only if $(\infty,\ldots,\infty,i_j,\ldots,0)\in \Gamma_S$; so we have:
\begin{eqnarray*}
&&d(S((c_1,\ldots,c_{j-1},i_j,\ldots,0))\setminus S((c_1,\ldots,c_{j-1},i_j+1,\ldots,0)))=1 \Leftrightarrow\\
&&\{\bs{\beta}\in \Gamma_S\hspace{0.1cm}|\hspace{0.1cm} \beta_j=i_j\hspace{0.1cm},\hspace{0.1cm} \beta_k\geq 0,k>j \text{ and } \beta_k=\infty,k<j \hspace{0.1cm}\}\neq \emptyset \Leftrightarrow\\
&&\text{there exists an element of the form } (\infty,\infty,\ldots,i_j,\alpha_{j+1},\ldots,\alpha_d)\in \Gamma_S.  
\end{eqnarray*}
We observe $d(S((c_1,\ldots,c_{j-1},0,\ldots,0))\setminus S((c_1,\ldots,c_{j-1},1,\ldots,0)))=0$, otherwise there would be an element of the form $(c_1,\ldots,c_{j-1},0,\alpha_{j+1},\ldots,\alpha_d)\in S$, but this contradicts the locality of $S$.
Hence, by Proposition \ref{propBDF}.1), we have:
$$d(S((c_1,\ldots,c_{j-1},1,\ldots,0))\setminus S((c_1,\ldots,c_{j-1},c_j,\ldots,0)))=l_S(S^{\infty(U_j)}_j),$$
Using again  Proposition \ref{propBDF}.1), by the definition of length we obtain:
$$l(S)=\sum_{i=1}^d l_S(S^{\infty(U_i)}_i).$$
The proof of the second formula is analogous to the first one.
\end{proof}
\begin{ex}
\label{countlenght}
Let us consider the good semigroup $S\subseteq \N^4$ having small elements:
\begin{eqnarray*}
\operatorname{Small}(S)&=&\{ (2, 3, 2, 9 ), (2, 4, 4, 9), (2, 4, 6, 9 ), (2, 4, 8, 9 ), ( 2, 4, 9, 9 ), \\
&&  ( 2, 4, 10, 9), (4, 3, 2, 10), (4, 6, 4, 18 ), ( 4, 7, 6, 18), ( 4, 7, 8, 18 ),  \\ &&(4, 7, 9, 18 ), (4, 7, 10, 18),
 (4, 8, 6, 18 ), (4, 8, 8, 18), (4, 8, 9, 18), \\ &&(4, 8, 10, 18 ), ( 6, 3, 2, 10 ), ( 6, 6, 4, 18 ),
 ( 6, 7, 6, 18 ), ( 6, 7, 8, 18 ),\\ && ( 6, 7, 9, 18 ), ( 6, 7, 10, 18 ), ( 6, 9, 6, 18 ), ( 6, 10, 8, 18 ), ( 6, 10, 9, 18 ),  \\&& (6, 10, 10, 18 ), ( 7, 3, 2, 10 ), ( 7, 6, 4, 18 ), ( 7, 7, 6, 18 ), ( 7, 7, 8, 18 ),  \\
&&( 7, 7, 9, 18 ), ( 7, 9, 6, 18 ), ( 7, 10, 8, 18 ), ( 7, 10, 9, 18 ), ( 8, 3, 2, 10 ), \\
&&( 8, 6, 4, 18 ), ( 8, 7, 6, 18 ), ( 8, 7, 8, 18 ), ( 8, 7, 10, 18 ), ( 8, 9, 6, 18 ), \\ && ( 8, 10, 8, 18 ), ( 8, 10, 10, 18 )\}
\end{eqnarray*}
The conductor is $\bs{c(S)}=(8,10,10,18)$. We compute the sets $S^{\infty(U_i)}_i$.
$$S^{\infty(U_1)}=S$$ $$S_1^{\infty(U_1)}=\{0,2,4,6,7,8=c_1,\rightarrow\}=S_1.$$
$l_S(S_1^{\infty(U_1)})=5$, $g_S(S_1^{\infty(U_1)})=3$.
\begin{eqnarray*}
S^{\infty(U_2)}&=&\{( \infty, 3, 2, 10 ), ( \infty, 6, 4, 18 ),
( \infty, 7, 6, 18 ), ( \infty, 7, 8, 18 ),\\ 
&& ( \infty, 7, 10, 18 ), ( \infty, 9, 6, 18 ), ( \infty, 10, 8, 18 ), ( \infty, 10, 10, 18 ), \ldots\}
\end{eqnarray*}
$S_2^{\infty(U_2)}=\{3,6,7,9,10=c_2,\rightarrow\}$, $l_S(S_2^{\infty(U_2)})=4$, $g_S(S_2^{\infty(U_2)})=6$.
\begin{eqnarray*}
S^{\infty(U_3)}&=&\{( \infty, \infty, 8, 18 ), ( \infty, \infty, 10, 18 ), \ldots\}
\end{eqnarray*}
$S_3^{\infty(U_3)}=\{8,10=c_3,\rightarrow\}$, $l_S(S_3^{\infty(U_3)})=1$, $g_S(S_3^{\infty(U_3)})=9$.
\begin{eqnarray*}
S^{\infty(U_4)}&=&\{( \infty, \infty, \infty, 18 ),\ldots\}
\end{eqnarray*}
$S_4^{\infty(U_4)}=\{18=c_4,\rightarrow\}$, $l_S(S_4^{\infty(U_4)})=0$, $g_S(S_4^{\infty(U_4)})=18$.\\
\begin{eqnarray*}
l(S)&=&l_S(S_1^{\infty(U_1)})+l_S(S_2^{\infty(U_2)})+l_S(S_3^{\infty(U_3)})+l_S(S_4^{\infty(U_4)})\\ &=&5+4+1+0=10\\
g(S)&=&g_S(S_1^{\infty(U_1)})+g_S(S_2^{\infty(U_2)})+g_S(S_3^{\infty(U_3)})+g_S(S_4^{\infty(U_4)})\\&=&3+6+9+18=36
\end{eqnarray*}
\end{ex}
\vspace{0.2cm}
\textbf{Notations:}\\
In case of good semigroups with two branches, we can write in a different way the formula for the computation of the genus.
Now we introduce some notations regarding the good semigroups of $\N^2$ that will be useful also in the next section.
Given a good semigroup $S\subseteq \N^d$ and an element $\bs{\alpha}\in \N^d$, following the notation in \cite{anal:unr}, we set:
\begin{eqnarray*}
	\Delta_i(\bs{\alpha})&:=&\{\bs{\beta}\in \Z^{d}|\alpha_i=\beta_i \text{ and } \alpha_j<\beta_j \text{ for } j\neq i\}\\
	\Delta(\bs{\alpha})&:=&\bigcup_{i=1}^d\Delta_i(\bs{\alpha})\\
	\Delta_i^S(\bs{\alpha})&:=&S\cap \Delta_i(\bs{\alpha})\\
	\Delta^S(\bs{\alpha})&:=&S\cap\Delta(\bs{\alpha}).
\end{eqnarray*}

Furthermore we define:
\begin{eqnarray*} {}_i\Delta(\bs{\alpha})&:=&\{\bs{\beta}\in \Z^{d}|\alpha_i=\beta_i \text{ and } \beta_j<\alpha_j \text{ for } j\neq i\}\\
	_i\Delta^S(\bs{\alpha})&:=&S\cap _i\Delta(\bs{\alpha}).\end{eqnarray*}

Extending some of the previous definitions to infinite elements of $\overline{\N}^2$, we set
\begin{eqnarray*}
	_1\Delta((\alpha_1,\infty))&:=&\{\bs{\beta}\in \Z^{2}|\beta_1=\alpha_1 \}\\
		_2\Delta((\alpha_1,\infty))&:=&\emptyset\\
	_1\Delta((\infty,\alpha_2))&:=&\emptyset\\
		_2\Delta((\infty,\alpha_2))&:=&\{\bs{\beta}\in \Z^{2}|\beta_2=\alpha_2 \}
		\\
	_i\Delta^S(\bs{\alpha})&:=&S\cap {}_i\Delta(\bs{\alpha}).
\end{eqnarray*}

From this point onwards, with a little abuse of notation, we denote again by $S$ the semiring associated to the good semigroup $S$.

We will say that an element $\bs{\alpha}\in S \setminus{\bs{0}}$ is \emph{irreducible} if, from $\bs{\alpha}=\bs{\beta}+\bs{\gamma}$ with $\bs{\beta},\bs{\gamma}\in S$, it follows $\bs{\alpha}=\bs{\beta}$ or $\bs{\alpha}=\bs{\gamma}$. An element that is not irreducible will be called \emph{reducible}.
We denote by $I(S)$ the set of irreducible elements of $S$.
An element $\bs{\alpha}\in S$ will be called \textit{maximal} in $S$ if $\bs{\alpha}\in S\setminus S^{\infty}$ and $\Delta_S(\bs{\alpha})=\emptyset$ (finite maximal), or if $\bs{\alpha}\in S^{\infty}$ (infinite maximal).

We denote by $A_f(S)$ the set of \emph{finite maximals} in $S$ and by $A(S):=A_f(S)\cup S^{\infty}$ the set of all maximals in $S$.
Furthermore we call $I_A(S)$ the set of all \emph{irreducible maximals} in $S$.
It is easy to see that the set of irreducible maximals is finite and in \cite{emb-NG} it is proved that $S$ is generated by these elements as a semiring.

\begin{cor}
\label{delgform}
If $S\subseteq \N^2$ is a good semigroup, then
$g(S)=g(S_1)+g(S_2)+|A_f(S)|.$
\end{cor}
\begin{proof}
In this case, by the last proposition, it follows that
$g(S)=g(S_1)+g_S(S^{\infty(2)}_2)$.
We notice that:
$$\N\setminus S^{\infty(2)}_2=(\N\setminus S_2)\cup\{y\hspace{0.1cm}|\hspace{0.1cm} \max\{x \hspace{0.1cm}|\hspace{0.1cm}\exists (x,y)\in S\}<c_1\}$$
and observing that the second set has the same cardinality as $A_f(S)$, we obtain the thesis.
\end{proof}

\section{The tree of  good semigroups of $\N^2$ by genus}
\label{sec2}
In \cite{Bras-Amoros2008}, it is presented a method to compute all numerical semigroups up to a fixed genus building a tree where each new level is obtained removing  minimal generators larger than the Frobenius number from the semigroups of the previous level.

In this section we want to repeat the same process with the good semigroups, building a tree of local good semigroups of $\N^2$, where the $g$th level of the tree collects all local good semigroups with genus $g$.
In order to follow this idea, we will show that, in this case, the analogous of minimal generators are the \emph{tracks} of the good semigroup originally defined in \cite{emb-NG}.\\
Here we will recall the definition:

\begin{defi}
\label{pitrack}
Given $\bs{\alpha},\bs{\beta}\in I_A(S)$ we say that $\bs{\alpha}$ and $\bs{\beta}$ are connected by a \emph{piece of track} if they are not comparable, i.e. $\bs{\alpha} \not \leq \bs{\beta}$ and $\bs{\beta} \not \leq \bs{\alpha}$, and denoting $\bs{\gamma}=\bs{\alpha} \oplus \bs{\beta}$, we have $\Delta^S(\bs{\gamma})\cap (S\setminus I(S))=\emptyset$.
\end{defi}

\begin{defi}
\label{track}
Given $\bs{\alpha_1},\ldots,\bs{\alpha_n}\in I_A(S)$, with $\alpha_{11}<\ldots<\alpha_{n1}$ we say that $\bs{\alpha_1},\ldots,\bs{\alpha_n}$ are connected by a \emph{track} if 
we have:
\begin{itemize}
    \item $_2\Delta^S(\bs{\alpha_1})\cap(S\setminus I(S))=\emptyset$;
    \item $_1\Delta^S(\bs{\alpha_n})\cap(S\setminus I(S))=\emptyset$;
    \item $\bs{\alpha_i}$ and $\bs{\alpha_{i+1}}$ are connected by a \emph{piece of track} for all $i\in\{1,\ldots, n-1\}$.
\end{itemize}
In this case, denoted by $\bs{\gamma_i}=\bs{\alpha_i}\oplus \bs{\alpha_{i+1}}$ for $i\in\{1,\ldots, n-1\}$, we set:
$$T((\bs{\alpha_1},\ldots,\bs{\alpha_n}))=\{\bs{\alpha_1}\} \cup {}_2\Delta^S(\bs{\alpha_1}) \cup \left(\cup_{i=1}^{n-1}\Delta^S(\bs{\gamma_i})\right)\cup {}_1\Delta^S(\bs{\alpha_n}) \cup \{\bs{\alpha_n}\},$$
and we call this set  the \emph{track} connecting $\bs{\alpha_1},\ldots,\bs{\alpha_n}$.
\end{defi}
We will simply say that $T\subseteq S$ is  a \emph{track} in $S$ if there exist $\bs{\alpha_1},\ldots,\bs{\alpha_n}\in I_A(S)$ such that $T$ is the track connecting $\bs{\alpha_1},\ldots,\bs{\alpha_n}$. We call $\bs{\alpha_1}$ and $\bs{\alpha_n}$ respectively the \textit{starting point} and the \textit{ending point} of the track.
We include in the definition of tracks of $S$, also the sets of the form  $T((\bs{\alpha}))={}_1{\Delta}^S(\bs{\alpha})\cup\{\bs{\alpha}\}\cup{}_2{\Delta}^S(\bs{\alpha})$ consisting only of an element $\bs{\alpha}\in I_A(S)$ such that $_i\Delta^S(\bs{\alpha})\cap(S\setminus I(S))=\emptyset$, with $i=1,2$. In this case the starting point and the ending point are the same.
Noting that the previous definition implies that a track $T$ of $S$ never contains elements $\bs{\alpha}$ such that $\bs{\alpha} \geq \bs{c(S)}+\bs{e(S)}.$\\

The following statement holds:
\begin{lem}
\label{semcont}
Given a good semigroup $S$, and  a track $T=T((\bs{\alpha_1},\ldots,\bs{\alpha_n}))$ in $S$, then,
$S'=S\setminus T$ is a good semigroup strictly contained in $S$. 
\end{lem}
\begin{proof}
See \cite{emb-NG}.
\end{proof}
 
Given $A,B\subseteq \N^d$ we set: \[\min(A,B)=\{\min(\bs{\alpha},\bs{\beta})\hspace{0.1cm}|\hspace{0.1cm} \bs{\alpha}\in A, \bs{\beta}\in B\}. \]
In the next two theorems we will establish a relationship between the tracks and the genus of a good semigroup.

\begin{teo}
\label{teo1}
Each local good semigroup $S' \neq \N^2(1,1)$ with genus $g(S')$ can be obtained removing a track from a good semigroup with genus $g(S')-1$.
\end{teo}
\begin{proof}
Let us consider a semigroup $S'$ with genus $g(S')$; we have $\bs{c(S')}=(c_1,c_2)$ and $\bs{f(S')}=(c_1-1,c_2-1)=(f_1,f_2)$. We will distinguish two cases.\\
\textbf{Case 1:} $\bs{f(S')}\notin S'$. In this case we introduce the following sets:
\begin{align*}
    & &X=\{x\hspace{0.1cm}|\hspace{0.1cm}(x,f_2)\in S'\}\hspace{1cm}\textrm{ if }f_2 \neq 0, \\ 
    & &Y=\{y\hspace{0.1cm}|\hspace{0.1cm}(f_1,y)\in S'\}\hspace{1cm}\textrm{ if }f_1 \neq 0. 
\end{align*}
Notice that $S' \neq \N^2(1,1)$ implies that at least one among $X$ and $Y$ can be always considered.
If $X$ and $Y$ are defined and not empty we consider respectively:
$\tilde{x}=\max\{x\hspace{0.1cm}|\hspace{0.1cm} (x,f_2)\in S'\}$ and $\tilde{y}=\max\{y\hspace{0.1cm}|\hspace{0.1cm} (f_1,y)\in S'\}$.
If $X=\emptyset$, we denote by $$T_X=\Delta_2(\bs{f(S')})\cup [(\N^2\setminus S') \cap  \min{(S',\Delta_2(\bs{f(S')})})],$$ otherwise we denote by $$T_X=\Delta_2(\bs{f(S')})\cup\Delta_1((\tilde{x},f_2))\cup [(\N^2\setminus S') \cap  \min{(S',\Delta_2(\bs{f(S')})})].$$
If $Y=\emptyset$, we denote by $$T_Y=\Delta_1(\bs{f(S')})\cup [(\N^2\setminus S') \cap  \min{(S',\Delta_1(\bs{f(S')})})],$$ otherwise we denote by $$T_Y=\Delta_1(\bs{f(S')})\cup\Delta_2(f_1,\tilde{y})\cup  [(\N^2\setminus S') \cap  \min{(S',\Delta_1(\bs{f(S')})})].$$
We consider the following sets: 

\begin{align*}
    S^1=S'\cup T_Y,\\
    S^2=S'\cup T_X.
\end{align*}
In order to prove the thesis we can reduce to consider only the good semigroup $S^2$, since considering $S^1$ the proof is analogous.\\
According to definition of tracks, it is easy to observe that $T_X$ is a track of $S^2$ having the form $T_X=T((\tilde{x},\infty),(\infty,f_2))$ if $X\neq \emptyset$ and $T_X=T((\infty,f_2))$ otherwise.\\
If $X=\emptyset$, by removing the track $T_X$, we have that $A_f(S^2)=A_f(S')$, $g(S^2_2)=g(S'_2)-1$ and $g(S^2_1)=g(S'_1)$.
As a consequence of Corollary \ref{delgform} we have that $S'$ is obtained removing a track from the good semigroup $S^2$ having genus $g(S')-1$.\\
If $X\neq \emptyset$, by removing the track $T_X$, we have that $|A_f(S^2)|=|A_f(S')|-1$, $g(S^2_2)=g(S'_2)$ and $g(S^2_1)=g(S'_1)$.
Hence, using again  Corollary \ref{delgform} we have that $S'$ is obtained removing a track from the good semigroup $S^2$ having genus $g(S')-1$.\\

	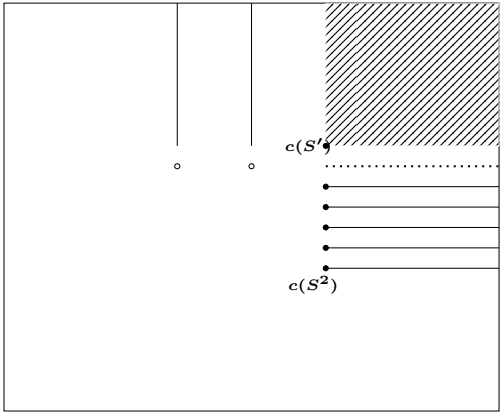
\begin{figure}[H]
	\pgfplotsset{ticks=none}
\tikzset{mark size=1}
\begin{tikzpicture}[scale=0.95]
	\begin{axis}[ xmin=0, ymin=0, xmax=20, ymax=20, ytick={0,7,11,13,20}, xtick={0,7,11,13,20}]
	\addplot[mark=none, black] coordinates {(13,7) (20,7)};
	\addplot[mark=none, black] coordinates {(13,8) (20,8)};
	\addplot[mark=none, black] coordinates {(13,9) (20,9)};
	\addplot[mark=none, black] coordinates {(13,10) (20,10)};
	\addplot[mark=none, black] coordinates {(13,11) (20,11)};
	\addplot[mark=none, black] coordinates {(10,13) (10,20)};
	\addplot[mark=none, black] coordinates {(7,13) (7,20)};
	\addplot[thick, style=dotted]coordinates{(13,12)(20,12)};
	\addplot [pattern = north east lines, draw=white]coordinates{(13,13)(20,13)(20,20)(13,20)(13,13)};
	
	\addplot[only marks, mark=o] coordinates{(7,12) (10,12)}; 
	\addplot[only marks] coordinates{(13,7) (13,8) (13,9) (13,10) (13,11) (13,13)}; 
	\node[font=\tiny] at (12.3,13) {$\bs{c(S')}$};
	\node[font=\tiny] at (12.5,6.2) {$\bs{c(S^2)}$};

\end{axis}
\end{tikzpicture}
\caption{Representation of $T_X$ in case $X=\emptyset$. The elements of $T_X$ are represented with a dashed line and white dots.}
	\end{figure}
	
		\begin{figure}[H]
	\pgfplotsset{ticks=none}
\tikzset{mark size=1}
\begin{tikzpicture}[scale=0.95]
	\begin{axis}[ xmin=0, ymin=0, xmax=20, ymax=20, ytick={0,7,11,13,20}, xtick={0,7,11,13,20}]
	\addplot[mark=none, black] coordinates {(13,7) (20,7)};
	\addplot[mark=none, black] coordinates {(13,8) (20,8)};
	\addplot[mark=none, black] coordinates {(13,9) (20,9)};
	\addplot[mark=none, black] coordinates {(13,10) (20,10)};
	\addplot[mark=none, black] coordinates {(13,11) (20,11)};
	\addplot[mark=none, black] coordinates {(11,13) (11,20)};
	\addplot[mark=none, black] coordinates {(9,13) (9,20)};
	\addplot[mark=none, thick, style=dotted] coordinates{(13,12)(20,12)};
	\addplot[mark=none, thick, style=dotted]coordinates{(8,13)(8,20)};
	\addplot [pattern = north east lines, draw=white]coordinates{(13,13)(20,13)(20,20)(13,20)(13,13)};
	
	\addplot[only marks] coordinates{(13,7) (13,8) (13,9) (13,10) (13,11) (8,12) (13,13)}; 
	\addplot[only marks, mark=o] coordinates{(9,12) (11,12)}; 
	\node[font=\tiny] at (12.3,13) {$\bs{c(S')}$};
	\node[font=\tiny] at (7.5,11.5) {$(\tilde{x},f_2)$};
	\node[font=\tiny] at (12.5,6.2) {$\bs{c(S^2)}$};

\end{axis}
\end{tikzpicture}
\caption{Representation of $T_X$ in case $X\neq\emptyset$. The elements of $T_X$ are represented with dashed lines and white dots.}
	\end{figure}
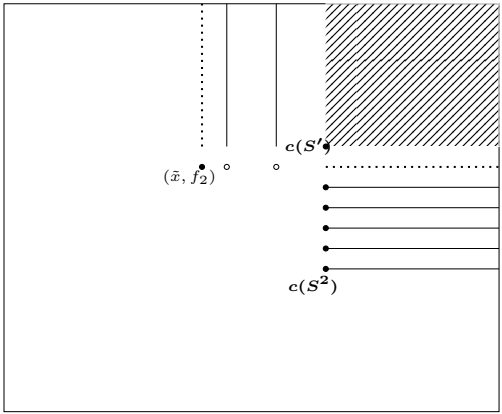

\textbf{Case 2:} $\bs{f(S')}\in S'$. In this case we define $T=\Delta(\bs{f(S')})$. If we consider \[S=S'\cup T;\] also in this case it is easy to observe that $S$ is a good semigroup and\\ $T=T((f_1,\infty),(\infty,f_2))$ is a track of $S$.

Removing the track $T$ from $S$, we have that $|A_f(S)|=|A_f(S')|-1$, $g(S_2)=g(S'_2)$ and $g(S_1)=g(S'_1)$.
Hence, using again  Corollary \ref{delgform} we have that $S'$ is obtained removing a track by the good semigroup $S$ having genus $g(S')-1$.\\

	\begin{figure}[H]
		\begin{center}
	\tikzset{mark size=1}
	\pgfplotsset{ticks=none}
\tikzset{mark size=1}
\begin{tikzpicture}
	\begin{axis}[ xmin=0, ymin=0, xmax=20, ymax=20, ytick={0,7,11,13,20}, xtick={0,7,11,13,20}]
	\addplot[mark=none, thick, style=dotted]coordinates{(12,13)(12,20)};
	\addplot[mark=none, thick, style=dotted]coordinates{(13,12)(20,12)};
	\addplot[mark=none, black] coordinates {(13,7) (20,7)};
	\addplot[mark=none, black] coordinates {(13,8) (20,8)};
	\addplot[mark=none, black] coordinates {(13,9) (20,9)};
	\addplot[mark=none, black] coordinates {(13,10) (20,10)};
	\addplot[mark=none, black] coordinates {(13,11) (20,11)};
	\addplot[mark=none, black] coordinates {(7,13) (7,20)};
	\addplot[mark=none, black] coordinates {(8,13) (8,20)};
	\addplot[mark=none, black] coordinates {(9,13) (9,20)};
	\addplot[mark=none, black] coordinates {(10,13) (10,20)};
	\addplot[mark=none, black] coordinates {(11,13) (11,20)};

	\addplot [pattern = north east lines, draw=white]coordinates{(13,13)(20,13)(20,20)(13,20)(13,13)};
	\addplot[only marks] coordinates{(7,7) (7,8) (7,9) (7,10) (7,11) (7,12) (8,7) (8,8) (8,9) (8,10) (8,11) (8,12) (9,7) (9,8) (9,9) (9,10) (9,11) (9,12) (10,7) (10,8) (10,9) (10,10) (10,11) (10,12) (11,7) (11,8) (11,9) (11,10) (11,11) (11,12) (12,7) (12,8) (12,9) (12,10) (12,11) (13,7) (13,8) (13,9) (13,10) (13,11) (7,13) (8,13) (9,13) (10,13) (11,13) (12,12) (13,13)}; 
	\node[font=\tiny] at (12.5,12.5) {$\bs{c(S')}$};
	\node[font=\tiny] at (6.5,6.2) {$\bs{c(S)}$};

\end{axis}
\end{tikzpicture}
\caption{Representation of $T$. The elements of $T$ are represented with dashed lines.}
\end{center}
\end{figure}
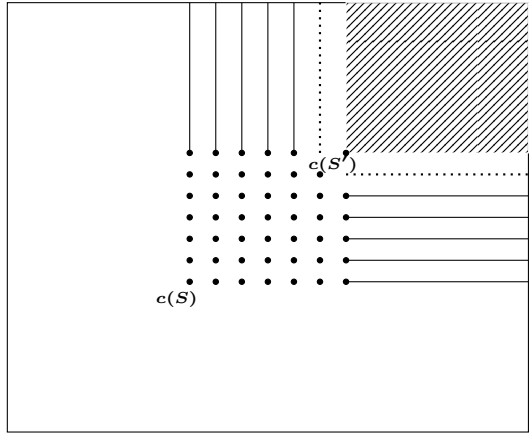
\end{proof}

\begin{teo}
\label{genuson}
Given a good semigroup $S$ with genus $g(S)$, removing a track of $S$ we obtain a semigroup with genus $g(S)+1$.
\end{teo}
\begin{proof}
Let $S$ be a good semigroup, we  consider a track $T:=T((\bs{\alpha_1},\ldots,\bs{\alpha_n}))$. With the same notation used in Lemma \ref{semcont}, we consider $S'=S\setminus T$. We want to prove that $g(S')=g(S)+1$.
We observe that by Definition \ref{pitrack}, $\bs{\alpha_i}\in A_f(S)$, for any $i\in\{2,\ldots,n-1\}$.\\
Let us consider first the case $\bs{\alpha_1}=(x,\infty)$. Now we can distinguish two cases:\\
\textbf{Case 1:} $\bs{\alpha_n}\in A_f(S)$. In this situation, removing the track $T$, we remove the elements $\bs{\alpha_i}$ with $i=1,\ldots,n$, hence we loose $n-1$ finite maximals, in $S'$. At the same time the elements $\bs{\gamma_i}=\min\{\bs{\alpha_i},\bs{\alpha_{i+1}}\}\in S'$ with $i=1,\ldots, n-1$ become finite maximals of $S'$; therefore $|A_f(S)|=|A_f(S')|$. If $\bs{\beta}\in {}_1\Delta^S(\bs{\alpha_n})$, by property (G3), there exists an element $\bs{\delta}\in \Delta^S_1(\bs{\beta})\cap S'$; clearly we have $\delta_2=\beta_2$. Furthermore we have $\alpha_{i,2}=\gamma_{i-1,2}$, for any $i\in\{2,\ldots,n\}$. Hence we proved that $g(S_2')=g(S_2)$.
Now we observe that $\Delta_1^{S'}(\bs{\alpha_n})=\emptyset$, since $S'\subset S$ and $\bs{\alpha_n}$ was a maximal of $S$, furthermore $_1\Delta^{S'}(\bs{\alpha_n})=\emptyset$, from the definition of $S'$. Hence in $S'$ we loose an element in the first projection; we have $g(S_1')=g(S_1)+1$.
By Corollary \ref{delgform}:
$$g(S')=g(S'_1)+g(S'_2)+|A_f(S')|=g(S_1)+1+g(S_2)+|A_f(S)|=g(S)+1.$$
\textbf{Case 2:} $\bs{\alpha_n}=(\infty,y)$. In this situation, by removing the track $T$, we remove the elements $\bs{\alpha_i}$ with $i=1,\ldots,n$, hence we loose $n-2$ finite maximals in $S'$. At the same time, the elements $\bs{\gamma_i}=\min\{\bs{\alpha_i},\bs{\alpha_{i+1}}\}$ with $i=1,\ldots, n-1$, become finite maximals in $S'$; hence $|A_f(S)|=|A_f(S')|-1$. Since $\alpha_{i,2}=\gamma_{i-1,2}$, for any $i\in\{2,\ldots,n\}$ and $\alpha_{i,1}=\gamma_{i,1}$, for any $i\in\{1,\ldots,n-1\}$; we have respectively $g(S_2')=g(S_2)$ and $g(S'_1)=g(S_1)$.
Therefore, again by Corollary \ref{delgform}:
$$g(S')=g(S'_1)+g(S'_2)+|A_f(S')|=g(S_1)+g(S_2)+|A_f(S)|+1=g(S)+1.$$

Now we suppose $\bs{\alpha_1}\in A_f(S)$; in this case, if $\bs{\alpha_n}=(\infty,y)$, the proof is analogous to the Case 1 that we have seen above, so we can suppose $\bs{\alpha_n}\in A_f(S)$.
In this case, by removing the track $T$, we remove the elements $\bs{\alpha_i}$, $i=1,\ldots,n$, hence we loose $n$ finite maximals in $S'$. At the same time the elements of $\gamma_i$ $i=1,\ldots, n-1$ become finite maximals in $S'$; hence $|A_f(S)|=|A_f(S')|+1$. With the same argument that we used in Case 1, it is easy to observe that $g(S_1')=g(S_1)+1$ and $g(S_2')=g(S_2)+1$. Hence we have:
$$g(S')=g(S'_1)+g(S'_2)+|A_f(S')|=g(S_1)+1+g(S_2)+1+|A_f(S)|-1=g(S)+1.$$
It is easy to notice that if the starting point and the ending point are the same the proof can be repeated in a similar way.
\end{proof}
We have observed that tracks play the same role as minimal generators in the case of numerical semigroups.
Now we want to show that, as in the numerical case, it is not necessary to consider all the minimal generators in order to build the tree. In fact, as we are going to show, it is sufficient to consider some special tracks.
\begin{defi} 
\label{specpar}
Given a local good semigroup $S\subseteq \N^2$, we say that  $T((\bs{\alpha_1},\ldots,\bs{\alpha_n}))$ is a \textit{beyond track}, if $\{\bs{\alpha}\in S, \bs{\alpha}\geq \bs{c(S)}\}\cap T((\bs{\alpha_1},\ldots,\bs{\alpha_n}))\neq \emptyset$.

Furthermore, we denote by $$\operatorname{BT}(S)=\{ T=T((\bs{\alpha_1},\ldots,\bs{\alpha_n}))\hspace{0.1cm} | \hspace{0.1cm} T \textrm{ is a beyond track of } S\}.$$
\end{defi}
If $S'$ is a good semigroup of $\N^2$ obtained from a good semigroup $S$ by removing a track, we say that $S$ is a \emph{parent} of $S'$ (or equivalently $S'$ is a \emph{son} of $S$). We say that $S$ is a \emph{special parent} of $S'$ (or equivalently $S'$ is a \emph{special son} of $S$), if $S'$ is obtained from $S$ removing a beyond track.
\begin{lem}
\label{condspec}
If $S$ is a special parent of $S'$, then $\bs{c(S)}<\bs{c(S')}.$
\end{lem}
\begin{proof}
Since $S$ is a special parent of $S'$, we have $S'=S\setminus T$, where $T$ is a beyond track of $S$. In particular we have $S'\subseteq S$, which implies that $\bs{c(S)}\leq \bs{c(S')}.$ Thus, we need to prove that $\bs{c(S)}\neq\bs{c(S')}$. Let us assume by contradiction that $\bs{c(S)}=\bs{c(S')}$.
Since $T$ is a beyond track of $S$, there exists an element $\bs{\beta}\in T$ with $\bs{\beta}\geq \bs{c(S)}$.
We recall that for a good semigroup $S'$ we have that $\bs{\alpha} \in S' \iff \min(\bs{\alpha},\bs{c(S')}) \in S'$. Thus, since $\bs{\beta}\notin S'$, we have $$\bs{c(S')}=\bs{c(S)}=\min(\bs{\beta},\bs{c(S)})=\min(\bs{\beta},\bs{c(S')}) \notin S',$$
which is a contradiction.
\end{proof}
 By Theorem \ref{teo1} we can deduce the following corollary.
\begin{cor}
\label{specparcar}
Let $S'$ be a good semigroup $S'\neq \N^2(1,1)$, with $\bs{f(S')}=(f_1,f_2)$.
We denote by $p=|\{ i \in \{1,2\} \hspace{0.1cm}| \hspace{0.1cm} f_i \neq 0 \} |$.
\begin{enumerate}
\item[1)] If $\bs{f(S')}\in S'$, then $S'$ has exactly one special parent.
\item[2)] If $\bs{f(S')}\notin S'$, then $S'$ has exactly $p$ special parents. 
\end{enumerate}
\end{cor}

\begin{proof}
$1)$ If we define the set $T=\Delta(\bs{f(S')})$ as in the proof of  Theorem \ref{teo1}, the good semigroup $S=S'\cup T$ is trivially a special parent of $S$.
Now we want to prove that, if there exists a good semigroup $\overline{S}$ and a beyond track $\overline{T}$ of $\overline{S}$, such that $S'=\overline{S}\setminus \overline{T}$, then $\overline{T}=T$.\\
Since $\bs{f(S')}\in S'$ then, by the previous lemma, there exists $\bs{\beta}\in \Delta(\bs{f(S')})\cap \overline{T}$. By property (G3) of good semigroups, it follows $\Delta(\bs{f(S')})\subseteq \overline{T}$. In this case, since $(f_1,\infty)$ and $(\infty,f_2)$, are respectively a starting point and an ending point of the track $\overline{T}$, we have $\overline{T}=\Delta(\bs{f(S')})=T$.

$2)$ If we define, when it is possible, the sets $T_X$ and $T_Y$ as in the proof of Theorem \ref{teo1}, the good semigroups $S^1$ and $S^2$ are trivially special parents of $S'$. Now we want to prove that, if $\overline{S}$ is a good semigroup and there exists a beyond track $\overline{T}$ of $\overline{S}$ such that $S'=\overline{S}\setminus \overline{T}$, then $\overline{T}=T_X$ or $\overline{T}=T_Y$.\\
By the previous lemma, there exists $\bs{\beta}\in \Delta(\bs{f(S')})\cap \overline{T}$. Thus, by property (G3) of good semigroups, it follows either $\Delta_1(\bs{f(S')})\subseteq \overline{T}$ or $\Delta_2(\bs{f(S')})\subseteq \overline{T}$. If we suppose $\Delta_1(\bs{f(S')})\subseteq \overline{T}$, then we prove that $\overline{T}=T_Y$. We start observing that $(f_1,\infty)$ is necessarily the starting point of $\overline{T}$. If the set $Y$ is empty, then it is also the ending point, hence $\overline{T}=T((f_1,\infty))=\Delta_1(\bs{f(S')})=T_Y$.\\
If $Y\neq \emptyset$, we observe that since $\Delta_1(\bs{f(S')})\subseteq \overline{T}$, then $\overline{T}=T((f_1,\infty),(\infty, y))$, where $(\infty, y)\in I_A(\overline{S})$. We recall that, in the proof of Theorem \ref{teo1}, we denoted by $\tilde{y}=\max\{y\hspace{0.1cm}|\hspace{0.1cm} (f_1,y)\in S'\}$; by maximality of $\tilde{y}$, it follows $y=\tilde{y}$. Hence we have $\overline{T}=T((f_1,\infty),(\infty, \tilde{y}))=T_Y$.
Assuming $\Delta_2(\bs{f(S')})\subseteq \overline{T}$, repeating the same proof, it is easy to observe that $\overline{T}=T_X$.
\end{proof}

Now we can introduce the tree of local good semigroups with a fixed genus.\\
\bf Notations: \rm\\
Now, we denote by $\mathfrak{N}_g$ the set of good semigroups with genus $g$.
We build the following family of sets of good semigroups:
\begin{itemize}
    \item $A_1=\N^2(1,1)$;
    \item $A_{i}=\{S'\hspace{0.1cm}|\hspace{0.1cm} S'\text{, is a special son of } S\in A_{i-1}\}$. 
\end{itemize}
We want to show that all semigroups of genus $g+1$ are special sons of semigroups of genus $g$, in other words:
\begin{prop}
$A_g=\mathfrak{N}_g$, for all $g\in \N\setminus \{0\}$.
\end{prop}
\begin{proof}
We work by induction on $g$. We suppose $A_{g-1}=\mathfrak{N}_{g-1}$ and we prove that $A_{g}=\mathfrak{N}_{g}$. Immediately by \ref{genuson}, $A_{g}\subseteq \mathfrak{N}_{g}$. If $S'\in \mathfrak{N}_{g}$, by Proposition \ref{specparcar}, there exists a special parent $S$ which, by Theorem \ref{genuson}, belongs to $\mathfrak{N}_{g-1}=A_{g-1}$; hence $S'\in A_{g}$.
\end{proof}
We proved that all good semigroups of genus $g$ can be obtained removing a special track from semigroups of genus $g-1$.
Hence, as a consequence of the previous proposition, the succession $A_1,A_2,\ldots,A_g$ defines a tree of local good semigroups in which the level $A_i$ consists of all good semigroups having genus $g$.
\begin{oss} \label{favson}
We denote by $n_g$ the number of local good semigroups of genus $g$. 
The fact that a good semigroup of $\mathfrak{N}_g$ can have two distinct special parents in $\mathfrak{N}_{g-1}$ implies that the formula
$$ n_g=\sum_{S \in \mathfrak{N}_{g-1}}{|\{ S'\hspace{0.1cm}|\hspace{0.1cm}S' \textrm{ is a special son of } S\} |}=\sum_{S \in \mathfrak{N}_{g-1}}{|\operatorname{BT}(S)|},$$
does not hold in general.

From the computational point of view, it would be convenient to determine, for each good semigroup $S$, a subset $T(S)\subseteq \operatorname{BT}(S) $ such that
$$ n_g=\sum_{S \in \mathfrak{N}_{g-1}}{|T(S)|},$$
for each $g \in \N.$

In order to do that, given a good semigroup $S$ with conductor $\bs{c(S)}=(c_1,c_2)$, we define

 \[
T(S)=
\left\{
\begin{aligned}
&\operatorname{BT}(S), &\textrm{ if } c_2=1 \\
&\{ T((\bs{\alpha_1},\ldots,\bs{\alpha_n}))\in \operatorname{BT}(S) \hspace{0.1cm}|\hspace{0.1cm} \bs{\alpha_n} \geq (\infty,c_2) \} &\textrm{ if }  c_2 \neq 1
\end{aligned}
\right.
\]

and we claim that $T(S)$ satisfies the required property.

It suffices to show that each good semigroup $S'$ has one and only one special parent $S$, such that $S'=S\setminus T$ with $T \in T(S)$.
Thus, let us consider an arbitrary good semigroup $S'$, with conductor $\bs{c(S')}=(c'_1,c'_2)$.\\
\textbf{Case 1:}  $c'_2=1$.
Corollary \ref{specparcar} tells us that $S'$ has only one special parent $S$, such that $S'=S \setminus T$, with $T \in \operatorname{BT}(S)$. Furthermore, denoting by $\bs{c(S)}=(c_1,c_2)$, we still have $c_2=1$. Hence, by definition of $T(S)$, we have $\operatorname{BT}(S)=T(S)$ and $T \in T(S)$ as required.
\\
\textbf{Case 2:}  $c'_2 \neq 1$ and $\bs{f(S')}=(f'_1,f'_2) \in S'$.
Corollary \ref{specparcar} tells us that $S'$ has only one special parent $S$, namely $S=S' \cup \Delta(\bs{f(S')})$. Notice that $\Delta(\bs{f(S')})$ is the track $T=T((f'_1,\infty),(\infty,f'_2)) \in \operatorname{BT}(S)$. Since we have $\bs{c(S)}\leq \bs{f(S')}$, it follows that $(\infty,f'_2) \geq (\infty,c_2), $ implying that $T \in T(S)$.
\\
\textbf{Case 3:}  $c'_2 \neq 1$ and $\bs{f(S')}=(f'_1,f'_2) \notin S'$.
In this case we have exactly two special parents 
\begin{align*}
    S^1=S'\cup T_Y,\\
    S^2=S'\cup T_X,
\end{align*}
where again considering the notations of \ref{teo1}
\[
T_Y= \left\{
\begin{aligned}
& T((f'_1,\infty)) \in \operatorname{BT}(S^1) &\textrm{ if } Y=\emptyset\\ 
& T((f'_1,\infty),(\infty,\tilde{y}))\in \operatorname{BT}(S^1)& \textrm{ if } Y\neq \emptyset
\end{aligned}
\right.
\]

\[
T_X= \left\{
\begin{aligned}
&T((\infty,f'_2))\in \operatorname{BT}(S^2) &\textrm{ if } X=\emptyset\\ 
&T((\tilde{x},\infty),(\infty,f'_2))\in \operatorname{BT}(S^2) &\textrm{ if } X\neq \emptyset
\end{aligned}
\right.
\]

Notice that $T_Y$ never belongs to $T(S^1)$, since $\bs{c(S^1)}=(c'_1-m,c'_2)$ with $m\geq 1$ and $(\infty,\tilde{y})<(\infty,c'_2).$ 
On the other hand, $T_X$ always belongs to $T(S^2)$ because, in both the possible definitions, the ending point is $(\infty,f'_2) \geq (\infty,\bs{c(S^2)}_2)$ (we recall that $\bs{c(S^2)}$ has the form $(c'_1,c'_2-n)$ with $n\geq 1$).
Thus $S^2$ is the required unique special parent of $S'$.

\end{oss}

Now we show with an example the construction of the tree up to genus $g=4$.
\begin{ex}
The starting point is the set
$$\mathfrak{N}_{1}=\{\N^2(1,1)\}.$$
The beyond tracks of $\N^2(1,1)$ are: $T_1=T((\infty,1))$, $T_2=T((1,\infty))$, $T_3=T((1,\infty),(\infty,1))$. If we denote by $S_i$ the special son associated to $T_i$, we have: $S_1=\{(0,0),(1,2)\}$, $S_2=\{(0,0),(2,1)\}$, $S_3=\{(0,0), (1,1), (2,2)\}$.\\
$$\mathfrak{N}_{2}=\{S_1,S_2,S_3\}.$$
Let us consider $S_1=\{(0,0),(1,2)\}$, we have the beyond tracks:

$T_{1,1}=T(( 1, \infty), (\infty, 2))$,  $T_{1,2}=T((1, \infty),(\infty,3))$, $T_{1,3}=T((1, \infty))$,
  $T_{1,4}=T((\infty, 2))$, $T_{1,5}=T((\infty, 3))$.
Notice that Remark \ref{favson} tells us that we can avoid to compute the son of $S_1$ with respect to the track $T_{1,3} \notin T(S_1)$.
We obtain:
$$S_{1,1}=\{(0,0), (1,2), (2,3)\},\hspace{0.5cm} S_{1,2}=\{(0,0), (1,2), (1,3), (2,2), (2,4)\},$$
$$S_{1,3}=\{(0,0), (2,2)\},\hspace{0.5cm} S_{1,4}=\{(0,0), (1,3)\},\hspace{0.5cm} S_{1,5}=\{(0,0), (1,2), (1,4)\}.$$
If we consider $S_2=\{(0,0), (2,1)\}$, we have the beyond tracks:

$T_{2,1}=T(( 2, \infty), (\infty, 1))$,  $T_{2,2}=T((3, \infty),(\infty,1))$, $T_{2,3}=T((\infty, 1))$,
  $T_{2,4}=T((2, \infty))$, $T_{2,5}=T((3, \infty))$.
We obtain:
$$S_{2,1}=\{(0,0), (2,1), (3,2)\},\hspace{0.5cm} S_{2,2}=\{(0,0), (2,1), (2,2), (3,1), (4,2)\},$$ $$S_{2,3}=\{(0,0), (2,2)\}\hspace{0.5cm} S_{2,4}=\{(0,0), (3,1)\},\hspace{0.5cm} S_{2,5}=\{(0,0), (2,1), (4,1)\}.$$
We notice that $S_{1,3}=S_{2,3}$, but in this case $T_{2,3}\in T(S_2)$.\\
For what concerns the good semigroup $S_3=\{ (0,0), (1,1), (2,2)\}$, we have only a beyond track:
$T_{3,1}=T(( 2, \infty), (\infty, 2))$.
We obtain:
$$S_{3,1}=\{(0,0), (1,1), (2,2), (3,3)\}.$$
$$\mathfrak{N}_{3}=\{S_{1,1},S_{1,2},S_{1,3},S_{1,4},S_{1,5},S_{2,1},S_{2,2},S_{2,3},S_{2,4},S_{2,5},S_{3,1}\}.$$
Now, we have:
\begin{itemize}
\item $S_{3,1}$ has got only one beyond track, so only a special son $$S_{3,1,1}=\{(0,0), (1,1),(2,2),(3,3),(4,4)\}.$$
\item $S_{2,1}$ has got two beyond tracks, thus
two special sons 
$$S_{2,1,1}=\{(0,0), (2,1),(4,2)\},\hspace{0.1cm} S_{2,1,2}=\{(0,0), (2,1),(3,2),(3,3),(4,2),(5,3)\}.$$
\item $S_{2,2}$ has no beyond tracks.
\item $S_{2,3}$ has got eight tracks. So it has got eight special sons:
$$S_{2,3,1}=\{(0,0), (2,2),(3,3)\},\hspace{0.1cm} S_{2,3,2}=\{(0,0), (2,2),(2,3),(3,2),(3,4)\},$$
$$S_{2,3,3}=\{(0,0), (2,2), (2,3), (2,4), (3,2), (3,3), (4,2), (4,4)\},$$ $$S_{2,3,4}=\{(0,0), (2,2), (2,3), (3,2), (4,3)\},\hspace{0.1cm} S_{2,3,5}=\{(0,0), (3,2)\},$$ $$S_{2,3,6}=\{(0,0), (2,2),(4,2)\},\hspace{0.1cm} S_{2,3,7}=\{(0,0), (2,3)\}, $$ $$\hspace{0.1cm} S_{2,3,8}=\{(0,0), (2,2),(2,4)\}.$$
\item $S_{2,4}$ has got seven tracks. So it has got seven special sons: $$S_{2,4,1}=\{(0,0), (3,1),(4,2)\},\hspace{0.1cm} S_{2,4,2}=\{(0,0), (4,1)\},$$ $$S_{2,4,3}=\{(0,0), (3,1),(3,2),(4,1),(5,2)\},\hspace{0.1cm} S_{2,4,4}=\{(0,0), (3,1),(5,1)\},$$ $$S_{2,4,5}=\{(0,0), (3,1),(3,2),(4,1),(5,2)\},$$ $$S_{2,4,6}=\{(0,0), (3,1),(4,1),(6,1)\},\hspace{0.1cm}S_{2,4,7}=\{(0,0), (3,2)\}.$$
\item $S_{2,5}$ has got four tracks. So it has got four special sons: 
$$S_{2,5,1}=\{(0,0), (2,1),(4,2)\},$$ $$S_{2,5,2}=\{(0,0), (2,1),(2,2),(4,1), (4,2), (5,1), (6,2)\},$$ $$S_{2,5,3}=\{(0,0), (2,1),(4,1),(6,1)\}, \hspace{0.1cm} S_{2,5,4}=\{(0,0), (2,2),(4,2)\}.$$
\item $S_{1,1}$ has got two tracks. So it has got two special sons: $$S_{1,1,1}=\{(0,0), (1,2),(2,3),(2,4),(3,3),(3,5)\},\hspace{0.1cm} S_{1,1,2}=\{(0,0), (1,2),(2,4)\}$$
\item $S_{1,2}$ has no tracks.
\item $S_{1,3}=S_{2,3}$ was studied before.
\item $S_{1,4}$ has got seven tracks. So it has got seven special sons:
$$S_{1,4,1}=\{(0,0), (1,3),(2,4)\},\hspace{0.1cm} S_{1,4,2}=\{(0,0), (1,3),(1,4),(2,3),(2,5)\},$$
$$S_{1,4,3}=\{(0,0), (1,3),(1,4),(1,5),(2,3),(2,4),(2,6)\},\hspace{0.1cm} S_{1,4,4}=\{(0,0), (2,3)\},$$
$$S_{1,4,5}=\{(0,0), (1,4)\},\hspace{0.1cm} S_{1,4,6}=\{(0,0), (1,3),(1,5)\},$$ $$ \hspace{0.1cm} S_{1,4,7}=\{(0,0), (1,3),(1,4),(1,6)\}$$
\item $S_{1,5}$ has got four tracks. So it has got four special sons:
$$S_{1,5,1}=\{(0,0), (1,2),(2,4)\}$$ $$S_{1,5,2}=\{(0,0), (1,3),(1,4),(1,5),(2,2),(2,4),(2,6)\}$$ $$S_{1,5,3}=\{(0,0), (2,2),(2,4)\},\hspace{0.1cm} S_{1,5,4}=\{(0,0), (1,2),(1,4),(1,6)\}.$$
\end{itemize}
Notice that 
\begin{itemize}
\item $S_{2,1,1}=S_{2,5,1}$, but $T_{2,1,1}\notin T(S_{2,1})$ while 
$T_{2,5,1} \in T(S_{2,5});$
\item $S_{2,3,5}=S_{2,4,7}$, but $T_{2,3,5}\notin T(S_{2,3})$ while 
$T_{2,4,7} \in T(S_{2,4});$
\item $S_{2,3,6}=S_{2,5,4}$, but $T_{2,3,6}\notin T(S_{2,3})$ while 
$T_{2,5,4} \in T(S_{2,5});$
\item $S_{1,4,4}=S_{2,3,7}$, but $T_{1,4,4}\notin T(S_{1,4})$ while 
$T_{2,3,7} \in T(S_{2,3});$
\item $S_{1,5,1}=S_{1,1,2}$, but $T_{1,5,1}\notin T(S_{1,5})$ while 
$T_{1,1,2} \in T(S_{1,1});$
\item $S_{1,5,3}=S_{2,3,8}$, but $T_{1,5,3}\notin T(S_{1,5})$ while 
$T_{2,3,8} \in T(S_{2,3}),$
\end{itemize}
thus, the repeated semigroups can be computed only one time by taking into account the  consequences of Remark \ref{favson}.
\end{ex}

\begin{figure}[H]
\begin{center}
	\begin{tikzpicture}
	[scale=0.75]
	\node (m1) at (0,26) {$n_1$};
	\node (m2) at (4,26) {$n_2$};
	\node (m3) at (8,26) {$n_3$};
    \node (m4) at (12,26) {$n_4$};
	\node (n0) at (0,10) {$\N(1,1)$};
	\node (n1) at (4,20)  {$S_1$};
	\node (n2) at (4,10) {$S_2$};
	\node (n3) at (4,0) {$S_3$};
	\node (n11) at (8,24)  {$S_{1,1}$};
	\node (n12) at (8,22)  {$S_{1,2}$};
	\node (n14) at (8,18)  {$S_{1,4}$};
	\node (n15) at (8,16) {$S_{1,5}$};
	\node (n21) at (8,14)  {$S_{2,1}$};
	\node (n22) at (8,12)  {$S_{2,2}$};
	\node (n23) at (8,10)  {$S_{2,3}$};
	\node (n24) at (8,8)  {$S_{2,4}$};
	\node (n25) at (8,6) {$S_{2,5}$}; 
	\node (n31) at (8,0) {$S_{3,1}$};
	\node (n111) at (12,24.25)  {$S_{1,1,1}$};
	\node (n112) at (12,23.75)  {$S_{1,1,2}$};
	\node (n141) at (12,21.75)  {$S_{1,4,1}$};
	\node (n142) at (12,21.25)  {$S_{1,4,2}$};
	\node (n143) at (12,20.75)  {$S_{1,4,3}$};
	\node (n145) at (12,19.75)  {$S_{1,4,5}$};
	\node (n146) at (12,19.25)  {$S_{1,4,6}$};
	\node (n147) at (12,18.75)  {$S_{1,4,7}$};
	\node (n152) at (12,17.25)  {$S_{1,5,2}$};
	\node (n154) at (12,16.25)  {$S_{1,5,4}$};
	\node (n212) at (12,14.75)  {$S_{2,1,2}$};
    \node (n231) at (12,12.75)  {$S_{2,3,1}$};
	\node (n232) at (12,12.25)  {$S_{2,3,2}$};
	\node (n234) at (12,11.25)  {$S_{2,3,3}$};
	\node (n235) at (12,10.75)  {$S_{2,3,4}$};
	\node (n237) at (12,9.75)  {$S_{2,3,7}$};
	\node (n238) at (12,9.25)  {$S_{2,3,8}$};
	\node (n241) at (12,8.25)  {$S_{2,4,1}$};
	\node (n242) at (12,7.75)  {$S_{2,4,2}$};
    \node (n243) at (12,7.25)  {$S_{2,4,3}$};
    \node (n244) at (12,6.75)  {$S_{2,4,4}$};
    \node (n245) at (12,6.25)  {$S_{2,4,5}$};
    \node (n246) at (12,5.75)  {$S_{2,4,6}$};
    \node (n247) at (12,5.25)  {$S_{2,4,7}$};
    \node (n251) at (12,4.25)  {$S_{2,5,1}$};
    \node (n252) at (12,3.75)  {$S_{2,5,2}$};
    \node (n253) at (12,3.25)  {$S_{2,5,3}$};
    \node (n254) at (12,2.75)  {$S_{2,5,4}$};
    \node (n311) at (12,0)  {$S_{3,1,1}$};
	\foreach \from/\to in {n0/n1,n0/n2,n0/n3,n1/n11,n1/n12,n1/n23,n1/n14,n1/n15,n2/n21,n2/n22,n2/n23,n2/n24,n2/n25,n3/n31,n11/n111,n11/n112,n14/n141,n14/n142,n14/n143,n14/n237,n14/n145,n14/n146,n14/n147,n15/n112,n15/n152,n15/n238,n15/n154,n21/n251,n21/n212,n23/n231,n23/n232,n23/n247,n23/n234,n23/n235,n23/n254,n23/n237,n23/n238,n24/n241,n24/n242,n24/n243,n24/n244,n24/n245,n24/n246,n24/n247,n25/n251,n25/n252,n25/n253,n25/n254, n31/n311}
	\draw (\from) -- (\to);
	\end{tikzpicture}
	\caption{Tree of local good semigroups until genus 4.}
	\end{center}
	\end{figure}
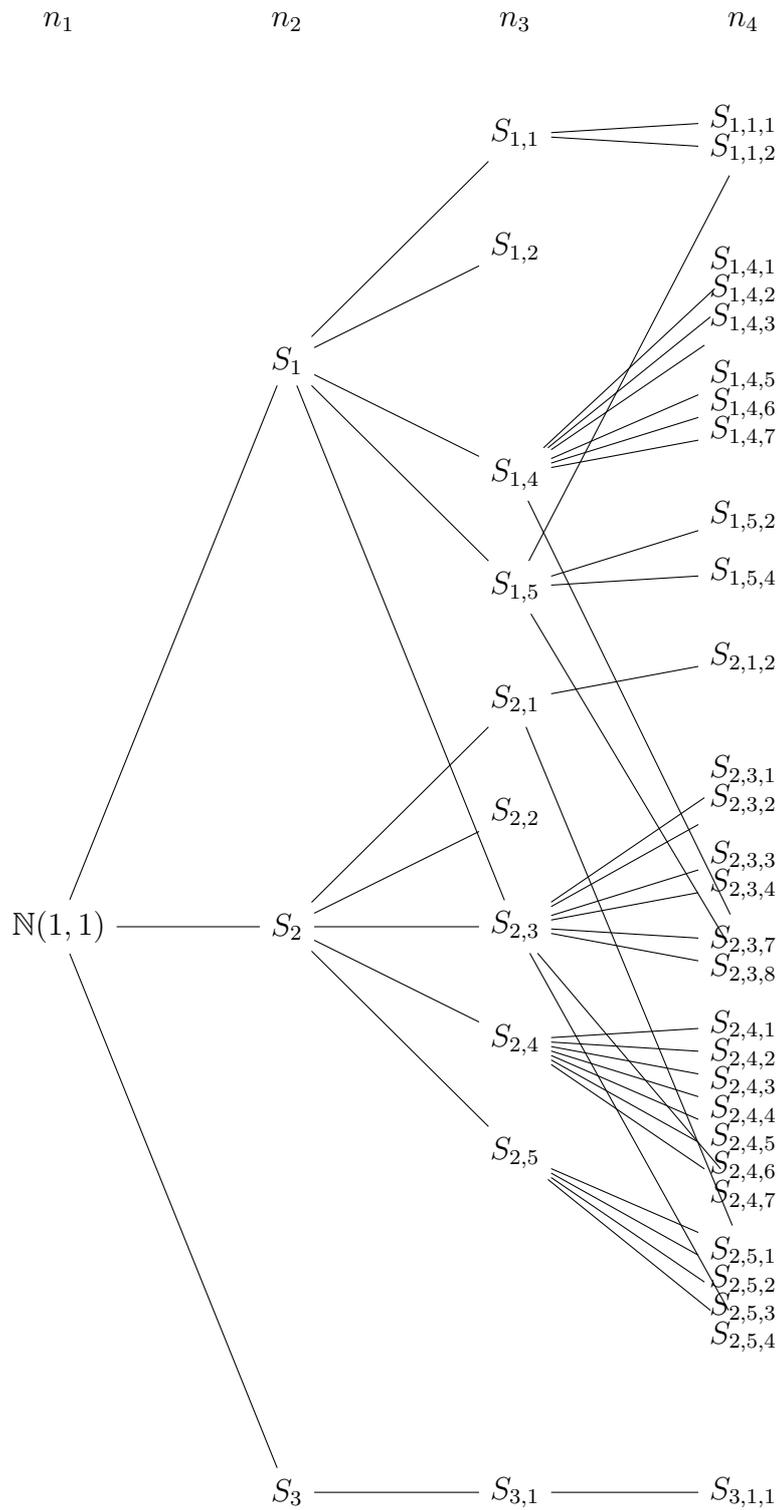

We implemented a function in "GAP" \cite{GAP4} using the package "NumericalSgps" \cite{GAPNumericalSgps} to find the special tracks and an algorithm to build the tree of local good semigroups by genus (where each level is built starting from the previous one). The following table contains the results obtained regarding the value of $n_g$ up to genus $27$:

\begin{table}[H]
\centering
\begin{tabular}{|l|l|l|l|}
\hline
$g$&$n_g$         & $\frac{n_g}{n_{g-1}}$   &  $\frac{n_g}{n_{g-1}}-\frac{n_{g-1}}{n_{g-2}}$ \\ \hline
1  & 1            &                         &                        \\ \hline
2  & 3            & 3                       &                        \\ \hline
3  & 10           & 3,333333                &       +0,333333        \\ \hline
4  & 29           & 2,9                     &       -0,433333        \\ \hline
5  & 78           & 2,689655                &       -0,210345        \\ \hline
6  & 211          & 2,705128                &       +0,015473        \\ \hline
7  & 555          & 2,630332                &       -0,074796        \\ \hline
8  & 1419         & 2,556757                &       -0,073575        \\ \hline
9  & 3658         & 2,577872                &       +0,021115        \\ \hline
10 & 9291         & 2,539913                &       -0,037959        \\ \hline
11 & 23559        & 2,53568                 &       -0,004233        \\ \hline
12 & 59750        & 2,536186                &       +0,000506        \\ \hline
13 & 151489       & 2,535381                &       -0,000805        \\ \hline
14 & 384962       & 2,541188                &       +0,005807        \\ \hline
15 & 981175       & 2,548758                &       +0,007570        \\ \hline
16 & 2509148      & 2,557289                &       +0,008531        \\ \hline
17 & 6446022      & 2,569008                &       +0,011719        \\ \hline
18 & 16643410     & 2,581966                &       +0,012958        \\ \hline
19 & 43206759     & 2,596028                &       +0,014062        \\ \hline
20 & 112813434    & 2,611014                &       +0,014986        \\ \hline
21 & 296385223    & 2,627216                &       +0,016202        \\ \hline
22 & 783663199    & 2,644069                &       +0,016854        \\ \hline
23 & 2085649918  & 2,661411                &       +0,017341       \\ \hline
24 & 5588023752   & 2,679272                &       +0,017861        \\ \hline
25 & 15074196720  & 2,697759                &       +0,018318        \\ \hline
26 & 40945190707  & 2,71624                 &       +0,018654       \\ \hline
27 & 111988822296 & 2,735091                &       +0,018847        \\ \hline
\end{tabular}
\caption{Table reporting the number of local good semigroups by genus}
\end{table}
In 2008 Bras Amorós observed that, in case of numerical semigroups the numerical sequence $\{n_g\}$ seemed to have the same behaviour of Fibonacci sequence and  conjectured that the ratio $\frac{n_g}{n_{g-1}}$ converges to the golden ratio \cite{Bras-Amoros2008}. In 2013 this was actually proved by Zhai \cite{zhai2013fibonacci}.
Looking at the previous table, even though we have not a great quantity of data, the tendency for local good semigroups of $\N^2$ appears to be different. In fact the difference $\frac{n_g}{n_{g-1}}-\frac{n_g-1}{n_{g-2}}$ seems to be an increasing function, hence the ratio seems to diverge.
\section{Relationship between genus and other notable elements}
\label{sec3}
\subsection{On the type of a good semigroup}
In this subsection we want to relate the genus and the type of a good semigroup $S\subseteq \N^2$ by generalizing a well known inequality that holds in the case of numerical semigroups.

First of all, we recall the concept of type of a good semigroup by following the definition introduced in \cite{type:good} that extends the one initially given in \cite{anal:unr}.

We write  $(\alpha_1,\alpha_2)\ll (\beta_1,\beta_2)$ if and only if $\alpha_1<\beta_1$ and $\alpha_2<\beta_2$. Now we can introduce a partial order on $S$ setting $(\alpha_1,\alpha_2)\leq\leq (\beta_1,\beta_2)$ if and only if $(\alpha_1,\alpha_2)=(\beta_1,\beta_2)$ or $(\alpha_1,\alpha_2)\neq (\beta_1,\beta_2)$ and $(\alpha_1,\alpha_2)\ll (\beta_1,\beta_2)$.
 
 Given a good semigroup $S \subseteq \N^2$, let us consider a set $A \subseteq S$ such that there exists $\bs{c} \in \N^2$ with $\bs{c}+\N^2 \subseteq S \setminus A$.
As described in \cite{type:good}, it is possible to build up a partition of such a set $A$, in the following way. 

Let us define, $D^{(0)}=\emptyset$:
\[B^{(i)}=\{\bs{\alpha}\in A \backslash (\bigcup_{j<i}D^{(j)}): \bs{\alpha} \text{ is maximal with respect to } \leq \leq\}\]
\[C^{(i)}=\{\bs{\alpha}\in B^{(i)}: \bs{\alpha}=\bs{\beta}_1\oplus \bs{\beta}_2 \text{ for some } \bs{\beta}_1,\bs{\beta}_2\in B^{(i)}\setminus\{\bs{\alpha}\}\}\]
\[D^{(i)}=B^{(i)}\backslash C^{(i)}.\]
For a certain $N\in \N$, we have $A=\bigcup_{i=1}^ND^{(i)}$ and $D^{(i)}\cap D^{(j)}=\emptyset$. According to the notations in \cite{type:good}, we rename these sets in an increasing order setting $A_i=D^{(N+1-i)}.$ Thus we have
$$A=\bigcup_{i=1}^NA_i,$$
and the sets $A_i$ are called \emph{levels} of $A$.

It was proved \cite[Thm. 2.7]{type:good} that, if $E=S \setminus A$ is a proper good ideal of $S$, then  $N=d(S\setminus E)$.
\begin{defi} \label{numberlevel}
Let us consider a set $A \subseteq \N^2$ such that there exists $\bs{c} \in \N^2$ with $\bs{c}+\N^2 \subseteq \N^2 \setminus A$.
We denote by $\operatorname{NL}(A)$ the integer such that 
$$ A=\bigcup_{i=1}^{\operatorname{NL}(A)}A_i,$$
is the partition in levels of $A$ described above.
\end{defi}
Now we want to generalize to good ideals a result proved for good semigroups in \cite{type:good}.
\begin{prop}
\label{levdis}
Let $I$ be a relative good ideal of a good semigroup $S\subseteq \N^2$. We consider $A\subseteq I$ satisfying the conditions of Definition \ref{numberlevel} and such that $E=I\setminus A$ is a relative good ideal of $S$. Then
$$\operatorname{NL}(A)=d(I\setminus E)$$
\end{prop}
\begin{proof}
If $I$ is a relative good ideal of $S$ there exists $\bs{\alpha}\in S$ such that $\bs{\alpha}+I\subseteq S$.
We notice that $(\bs{\alpha}+I)\cup\{0\}$ is a good semigroup. In fact, if $\bs{\alpha}+\bs{i_1},\bs{\alpha}+\bs{i_2}\in \bs{\alpha}+I$,
$$(\bs{\alpha}+\bs{i_1})+(\bs{\alpha}+\bs{i_2})=\bs{\alpha}+(\bs{\alpha}+\bs{i_1}+\bs{i_2})\in \bs{\alpha}+I$$
since $\bs{\alpha}+(\bs{i_1}+\bs{i_2})\in S+I\subseteq I$.
Furthermore, it is easy to check that $\bs{\alpha}+E$ is a proper good ideal of $(\bs{\alpha}+I)\cup\{0\}$. Setting $S'=(\bs{\alpha}+I)\cup\{0\}$, $E'=\bs{\alpha}+E$ and $A'=((\bs{\alpha}+I)\cup\{0\})\setminus (\bs{\alpha}+E)$, by \cite[Thm. 2.7]{type:good}, we have
$$d(S'\setminus E')=\operatorname{NL}(A')$$ 
so we can write:
$$1+d(\bs{\alpha}+I\setminus \bs{\alpha}+E)=d(S'\setminus E')=\operatorname{NL}(A')=1+\operatorname{NL}(\bs{\alpha}+I \setminus \bs{\alpha}+E).$$
We can conclude that:
$$d(I\setminus E)=d(\bs{\alpha}+I\setminus \bs{\alpha}+E)=\operatorname{NL}(\bs{\alpha}+I \setminus \bs{\alpha}+E)=\operatorname{NL}(I\setminus E)=\operatorname{NL}(A).$$
\end{proof}
Using this result we can easily rewrite the genus as the number of levels of $\N^2\setminus S$.

\begin{cor} \label{genlev}
Given a good semigroup $T$, $g(T)=\operatorname{NL}(\N^2\setminus T)$.
\end{cor}
\begin{proof}
It is sufficient to apply Proposition \ref{levdis}, considering $S=T$, $I=\N^2$, $A=\N^2\setminus T$.
\end{proof}

In \cite{type:good}, the set of
\emph{pseudo-frobenius elements}
of a good semigroup
$S$
is defined as $$ \operatorname{PF}(S)=\{\bs{\alpha} \in \N^2 \setminus S \hspace{0.1cm}| \hspace{0.1cm} \bs{\alpha}+M \subseteq S\},$$
where $M=S\setminus \{\bs{0}\}$ is the maximal ideal of $S$.

The set $\operatorname{PF}(S)$ satisfies the condition of the set $A$ in Definition \ref{numberlevel}. The type of the good semigroup $S$ is defined as $t(S)=\operatorname{NL}(\operatorname{PF}(S))$, that is  the number of levels of the pseudo-frobenius elements.
\begin{oss}
	We recall that, given two ideals $E$ and $F$ of $S$, it is possible to consider the set 
	$$ E-F=\{\bs{\alpha} \in \mathbb{Z}^2 \hspace{0.1cm}| \hspace{0.1cm} \bs{\alpha}+F \subseteq E\}.$$
	This set is not in general a good ideal.
	We have that $\operatorname{PF}(S)=(S-M)\setminus S$, thus $t(S)=\operatorname{NL}((S-M)\setminus S)$.
	In \cite[Proposition 3.5]{type:good}, it is proved that if $S-M$ is a good ideal, then $t(S)=d(S-M \setminus S)$ as it was initially defined in \cite{anal:unr}.
\end{oss}

Given a good semigroup $S$, we want to prove the inequality $$ g(S)\leq t(S)l(S), $$
that generalizes the analogous one  proved in 
\cite{froberg1986numerical} for numerical semigroups.

In order to do that we need the following lemma.
\begin{lem} \label{partiz}
Let us consider a subset $A \subseteq \N^2$ such that there exists $\bs{c} \in \N^2$ with $\bs{c}+\N^2 \subseteq \N^2 \setminus A$. Suppose that $A=\bigcup_{j=1}^hB_j$ with $B_l \cap B_m=\emptyset$ if $l\neq m$.
Then $$ \operatorname{NL}(A) \leq \sum_{j=1}^h{\operatorname{NL}(B_j)}.$$
\end{lem}
\begin{proof}

We denote by $n=\operatorname{NL}(A)$ and by $m_j=\operatorname{NL}(B_j)$. Furthermore we write 
$$ A=\bigcup_{i=1}^n{A_i}, \qquad B_j=\bigcup_{l=1}^{m_j}{B_{j,l}} \textrm{ for } j=1,\ldots,h.$$
We want to find  a chain $$ \bs{\alpha_1} \leq \ldots \leq\bs{\alpha_n}, 
\textrm{ where } \bs{\alpha_i} \in A_i \textrm{ for all }i=1,\ldots,n,$$ and  such that each $B_{j,l}$ contains at most one of the $\bs{\alpha_i}$'s.
In order to do that we consider $\bs{\alpha_1}$ as an arbitrary element of $A_1$, then we choose $\bs{\alpha_i}$, with $i\geq2$, by taking into account the following rule. We denote by $D=A_{i+1} \cap \{\bs{\beta} \in A| \bs{\alpha_i} \ll \bs{\beta} \}$.
\\
\textbf{Case 1:} If $D$ is not empty, then we choose as $\bs{\alpha_{i+1}}$ an arbitrary element of $D$.
\\
\textbf{Case 2:} If $D=\emptyset$, then \cite[Lemma 2.4 (1)]{type:good} ensures that $\Delta_1(\bs{\alpha_i}) \cap A_{i+1}$ and 
$\Delta_2(\bs{\alpha_i}) \cap A_{i+1}$ are both non-empty.
Furthermore, if we suppose that $\bs{\alpha_i} \in B_{j,l}$, then there must exist a $k \in \{1,2\}$ such that $\Delta_k(\bs{\alpha_i}) \cap B_{j,l}=\emptyset$. In fact, otherwise, we would have $\bs{\alpha_i}=\bs{\beta_1}\oplus \bs{\beta_2}$, with $\bs{\alpha_i},\bs{\beta_1},\bs{\beta_2} \in B_{j,l}$ that is a contradiction since $B_{j,l}$ is a level of $B_j$ and cannot contain such a configuration.
Thus, in this case we choose an element of $\Delta_k(\bs{\alpha_i}) \cap A_{i+1}$ as $\bs{\alpha_{i+1}}$.

By construction and by the properties of the levels, it is clear that it is not possible to find $l \in \{1,\ldots,m_j\}$ and $j \in \{ 1,\ldots,h\}$ such that $|B_{j,l} \cap \{\bs{\alpha_1}, \ldots, \bs{\alpha_n} \}| \geq 2.$ 

Thus $\operatorname{NL}(A) \leq \sum_{j=1}^h{\operatorname{NL}(B_j)}$ as required.
\end{proof}
Now we are ready to prove the main result of this subsection.
\begin{prop}
\label{boundsup}
	Let $S$ be a good semigroup. Then 
	$$ g(S) \leq t(S)l(S).$$
\end{prop}
\begin{proof}
	We denote by $n=l(S)$ and we choose $$ \bs{0}=\bs{\alpha_0}\leq \bs{\alpha_1}\leq \ldots \leq \bs{\alpha_n}=\bs{c(S)},$$
	an arbitrary saturated chain in $S$ between $\bs{0}$ and $\bs{c(S)}$. 
	We consider the following chain of ideals of $S$:
	$$ S \subseteq S-S(\bs{\alpha_1})=S-M \subseteq S-S(\bs{\alpha_2})\subseteq \ldots \subseteq S-S(\bs{\alpha_n})=S-C(S)=\N^2.$$
	
	We have $$ \N^2 \setminus S= \bigcup_{i=1}^n{(S-S(\bs{\alpha_i}))\setminus(S-S(\bs{\alpha_{i-1}}))},$$
	and by Lemma \ref{partiz} and Corollary \ref{genlev} we can deduce:
	\begin{equation} \label{ineq}
	g(S)= \operatorname{NL}(\N^2\setminus S) \leq \sum_{i=1}^n{\operatorname{NL}((S-S(\bs{\alpha_i}))\setminus(S-S(\bs{\alpha_{i-1}})))}.
	\end{equation} 
	Now we claim that \begin{eqnarray*}\operatorname{NL}((S-S(\bs{\alpha_i}))\setminus(S-S(\bs{\alpha_{i-1}}))) &\leq& \operatorname{NL}((S-S(\bs{\alpha_1}))\setminus(S-S(\bs{\alpha_{0}})))=\\ && \operatorname{NL}((S-M) \setminus S)=t(S),\end{eqnarray*} \textrm{ for all } $i=2,\ldots,n.$\\
	For each $i \in \{2,\ldots,n\},$ we denote by $$ 	\tilde{\bs{\alpha}}_{i-1}=\begin{cases}
\bs{\alpha_{i-1}}, &\textrm{ if } \bs{\alpha_i} \notin \Delta^S(\bs{\alpha_{i-1}}) \\
	 \max\{\Delta_{3-k}^S(\bs{\alpha_{i-1}})\}\oplus \bs{c(S)},& \textrm{ if } \bs{\alpha_i} \in \Delta_{k}^S(\bs{\alpha_{i-1}})  \end{cases},$$
	and we consider the following function
	\begin{eqnarray*}
                f:(S-S(\bs{\alpha_i}))\setminus(S-S(\bs{\alpha_{i-1}})) &\xhookrightarrow{}&  (S-M) \setminus S \\
                \bs{\gamma} &\to& \bs{\gamma}+\tilde{\bs{\alpha}}_{i-1}.
	\end{eqnarray*}
The function $f$ is clearly injective, thus in order to prove our claim we need only to show that it is well defined.

Thus we fix an arbitrary $\bs{\beta} \in (S-S(\bs{\alpha_i}))\setminus(S-S(\bs{\alpha_{i-1}}))$ and we prove that $\bs{\beta}+\tilde{\bs{\alpha}}_{i-1} \in (S-M) \setminus S$.
\begin{enumerate}
	\item We prove that $\bs{\beta}+\tilde{\bs{\alpha}}_{i-1} \in S-M$.
	Let us consider an element $\bs{\gamma}\in M=S(\bs{\alpha_1})$. We need to show that $\bs{\beta}+\tilde{\bs{\alpha}}_{i-1} +\bs{\gamma} \in S$.
    Since $\bs{\beta} \in S-S(\bs{\alpha_i})$, it suffices to notice that 
	$\tilde{\bs{\alpha}}_{i-1} +\bs{\gamma} \geq \bs{\alpha_i}$.
	In fact, otherwise $(\tilde{\bs{\alpha}}_{i-1} +\bs{\gamma} ) \oplus \bs{\alpha_i}=\bs{\delta} \in S$ would be an element such that $\bs{\alpha_{i-1}}< \bs{\delta} < \bs{\alpha_i}$, since $\bs{\alpha_{i-1}}\leq \tilde{\bs{\alpha}}_{i-1}\leq \tilde{\bs{\alpha}}_{i-1}+\bs{\gamma}$ and $\bs{\alpha_{i-1}}\leq \bs{\alpha_i}$. But this contradicts the fact
	that we considered a saturated chain in $S$.
	\item We prove that $\bs{\beta}+\tilde{\bs{\alpha}}_{i-1} \notin S$. Let us assume, by contradiction, that  $\bs{\beta}+\tilde{\bs{\alpha}}_{i-1} \in S$.
	We consider $D:=S(\bs{\alpha_{i-1}})\setminus S(\bs{\alpha_{i}})$.
	We have two cases:
	\\
	\textbf{Case 1:} $\bs{\alpha_i} \notin \Delta^S(\bs{\alpha_{i-1}})$.
	In this case we have $D=\{ \bs{\alpha_{i-1}}=\tilde{\bs{\alpha}}_{i-1}\}$. In fact, if $\bs{\beta} \in D$ and $ \bs{\beta}\neq \bs{\alpha_{i-1}}$, it would follow $ \bs{\alpha_{i-1}}< \bs{\beta}\oplus \bs{\alpha_i} < \bs{\alpha_i}$, against the fact that $\bs{\alpha_{i-1}}$ and $\bs{\alpha_{i}}$ are consecutive in $S$.
	Thus $\bs{\beta}+\tilde{\bs{\alpha}}_{i-1} \in S $ implies $\bs{\beta} \in (S-S(\bs{\alpha_{i-1}}))$ that is a contradiction.
	\\
	\textbf{Case 2:} There exists $k \in \{1,2\}$ such that $\bs{\alpha_i} \in \Delta_{k}^S(\bs{\alpha_{i-1}})$.
	It is easy to notice that in this case $D:=S(\bs{\alpha_{i-1}})\setminus S(\bs{\alpha_{i}})=\Delta_{3-k}^S(\bs{\alpha_{i-1}}) \cup \{\bs{\alpha_{i-1}}\}$.
	Let us consider an arbitrary $\bs{\gamma} \in D$ and let us show that $\bs{\beta}+\bs{\gamma} \in S$.
	If $\bs{\gamma}  \in \Delta_{3-k}^S( \tilde{\bs{\alpha}}_{i-1})$, then $\bs{\beta}+\bs{\gamma}  \in \Delta_{3-k}( \bs{\beta}+\tilde{\bs{\alpha}}_{i-1})$.
	We notice that, by definition of $\tilde{\bs{\alpha}}_{i-1}$, in this case $ \bs{c(S)} \in \Delta_{k}^S(\tilde{\bs{\alpha}}_{i-1})$. Since $\bs{\beta}+\tilde{\bs{\alpha}}_{i-1} \in S$, it easily follows $\bs{\beta}+\bs{\gamma}  \in \Delta_{3-k}(\bs{\beta}+\tilde{\bs{\alpha}}_{i-1}) \subseteq S$.
	If $\bs{\gamma} < \tilde{\bs{\alpha}}_{i-1}$, then property (G3) implies that there must exist a $\bs{\delta} \in \Delta_k^S(\bs{\gamma})$. Furthermore $\bs{\delta} \in S(\bs{\alpha_{i}})$, (otherwise $\bs{\alpha_{i-1}}< \bs{\delta}\oplus \bs{\alpha_i} < \bs{\alpha_i}$ ).
	Since $\bs{\beta} \in S-S(\bs{\alpha_i})$, we have $\bs{\beta}+\bs{\delta} \in S$.
	
	Finally we have:
	$$ \bs{\gamma}+\bs{\beta}=\overbrace{(\bs{\beta}+\bs{\delta})}^{\in S} \oplus \overbrace{(\bs{\beta}+\tilde{\bs{\alpha}}_{i-1})}^{\in S} \in S. $$
	Thus, also in this case, we deduce $\bs{\beta} \in (S-S(\bs{\alpha_{i-1}}))$ that is a contradiction.
 
\end{enumerate}	This means that $\bs{\beta}+\tilde{\bs{\alpha}}_{i-1} \notin S$ and the claim is proved.
Finally, by expression (\ref{ineq}) and the claim, it follows:
\begin{eqnarray*}
  g(S)&\leq&  \sum_{i=1}^n{\operatorname{NL}((S-S(\bs{\alpha_i}))\setminus(S-S(\bs{\alpha_{i-1}})))}\\ &\leq& \sum_{i=1}^{l(S)}(\operatorname{NL}((S-M)\setminus S))=t(S)l(S),
 \end{eqnarray*}
 and the proof is complete.
\end{proof}
Given a good semigroup $S$, its canonical ideal is the set $K=\{\bs{\alpha}\in \N^2|\Delta^S(\bs{\gamma}-\bs{\alpha})=\emptyset\}$. In \cite[Prop 2.17.]{anal:unr} it is proved that $g(S)\geq d(K\setminus C(S))$.\\
We observe that, if $S-M$ is a good ideal, $t(S)=d((S-M)\setminus S)$. Since $S-M\subseteq K\cup \Delta(\bs{\gamma})$, we have:
$$t(S)\leq d(K\cup \Delta(\bs{\gamma})\setminus S)\leq d(K\cup \Delta(\bs{\gamma})\setminus K)+d(K\setminus S)=1+d(K\setminus S)=$$
$$=d(K\setminus C(S))-d(S\setminus C(S))+1\leq g(S)-l(S)+1$$
Using Proposition \ref{levdis} it is possible to prove this inequality for any good semigroup $S$, also if $S-M$ is not a good semigroup:
\begin{cor}
\label{boundinf}
Given a good semigroup $S$, $g(S)\geq t(S)+l(S)-1$.
\end{cor}
\begin{proof}
Since $\operatorname{PF}(S)\subseteq (K\setminus S)\cup \Delta(\bs{\gamma})$, by \cite[Lemma 3.6]{type:good} and Lemma \ref{partiz}, we deduce
\begin{eqnarray*}
t(S)&=&\operatorname{NL}(\operatorname{PF}(S))\leq \operatorname{NL}((K\setminus S)\cup \Delta(\bs{\gamma}))\\&\leq& \operatorname{NL}((K\setminus S))+ \operatorname{NL}( \Delta(\bs{\gamma})) \leq  \operatorname{NL}(K\setminus S)+1
\end{eqnarray*}
If we apply the Proposition \ref{levdis}, considering $I=K$, $A=K\setminus S$, $E=S$, we have:
$$\operatorname{NL}(K\setminus S)+1=d(K\setminus S)+1=d(K\setminus C(S))-d(S\setminus C(S))+1\leq g(S)-l(S)+1.$$
\end{proof}
\subsection{On the Wilf Conjecture}
In \cite{wilf1978circle}, it was firstly introduced the well known Wilf conjecture regarding the numerical semigroups and in \cite{dobbs2003question} it was slightly rephrased. It states that the number of minimal generators of a numerical semigroup $S$, i.e. the embedding dimension of the semigroup, always satisfies the inequality
 $$ \operatorname{edim}(S) \geq \frac{c(S)}{c(S)-g(S)}, $$
 where $c(S)$ and $g(S)$ are respectively the conductor and the genus of the semigroup.
 
 The Wilf conjecture represents an important open problem in the context of the numerical semigroups theory, and it has been proved for many special cases \cite{froberg1986numerical}, \cite{kaplan2012counting}, \cite{sammartano2012numerical}, \cite{moscariello2015conjecture}, \cite{eliahou2017wilf}, \cite{bruns2019wilf}  and checked for numerical semigroups up to genus $50$ in (\cite{Bras-Amoros2008}) and up to genus $60$ in \cite{fromentin2016exploring}.
 
In \cite{emb-NG} the concept of embedding dimension of a good semigroup of $\N^2$ has been introduced, therefore, now it makes sense to extend in a natural way the conjecture to good semigroups.
Specifically we want to check if, for a good semigroup $S \subseteq \N^d$ of genus $g(S)$, the inequality 
$$ \operatorname{edim}(S) \geq \frac{c_S}{c_S-g(S)},$$
always holds, where $c_S$ is defined as in Section \ref{sec1}.

By exploring the tree of good semigroups of $\N^2$, introduced in the previous section, it is possible to check that the conjecture is satisfied for semigroups up to genus $22$.
However, starting from genus $23$, examples where the conjecture is not verified begin to show up, as we can see in the following case.
\begin{ex}
\label{countwilf}
We consider the good semigroup $S$, represented by the following set of small elements:
\begin{eqnarray*}\operatorname{Small}(S)&=&\{ (0,0), ( 4, 3 ), ( 8, 6 ), ( 8, 9 ), ( 8, 11 ), ( 8, 12 ), ( 8, 13 ), ( 8, 14 ), ( 9, 6 ),\\
&& ( 12, 9 ), ( 12, 11 ), 
  ( 12, 12 ),( 12, 13 ), ( 12, 14 ), ( 13, 9 ), ( 13, 11 ),\\
  && ( 15, 9 ), ( 16, 12 ), ( 16, 13 ), ( 16, 14 ),
  ( 17, 12 ), ( 17, 13 ), ( 17, 14 ),\\ 
  &&( 18, 12 ), ( 18, 13 ), ( 19, 12 ), ( 20, 14 )
\}. \end{eqnarray*}
We have $c_S=34$ and $g(S)=23$.
Using the algorithms presented in \cite{emb-NG}, it is possible to check that $\operatorname{edim}(S)=3$ (according to the terminology introduced in that paper, the set $\{(4,3),(8,\infty),(13,11)\}\subseteq I_A(S)$ constitutes a minimal set of representatives for $S$).

Finally we have: 
$$ 3=\operatorname{edim}(S)<\frac{c_S}{c_S-g(S)}=\frac{34}{11},$$
disproving the Wilf conjecture for good semigroups of $\N^2$.

\end{ex}
  \begin{oss}
  It still makes sense to ask whether the Wilf conjecture is true for good semigroups that are value semigroups.
  In fact, at the moment, there are no known examples of value semigroups disproving the conjecture, since for all the known counterexamples it seems impossible to find suitable rings having them as value semigroups.
  This fact may suggest that the Wilf conjecture is more related to the structure inherited from the rings than on the combinatorical properties of these objects.
  \end{oss}
\newpage
 \begin{acknowledgements}
	The authors would like to thank Marco D'Anna, Felix Delgado, Manuel Delgado and Pedro A. García Sánchez, for their helpful comments and suggestions during the development of this paper. Furthermore, they  want to thank the anonymous referees whose suggestions have been crucial to improve the work.
	The first author is supported by the projects MTM2014-55367-P, which is funded by Ministerio de Economía y Competitividad and Fondo Europeo de Desarrollo Regional FEDER, and by the Junta de Andalucía Grant Number FQM-343. He also thanks the "University of Granada" for hosting him and providing the machines necessary to complete some of the computations reported on the paper. Both the authors gratefully acknowledge support by the project "Proprietà algebriche locali e globali di anelli associati a curve e ipersuperfici" PTR 2016-18 - Dipartimento di Matematica e Informatica - Università di Catania".
\end{acknowledgements}

	\end{document}